\documentclass[reqno,a4paper,11pt]{amsart}
\usepackage{amsfonts,amsmath,amsthm,amssymb,amscd}
\usepackage{anyfontsize}

\usepackage{chngcntr}
\usepackage{comment}
\usepackage[all]{xy}
\usepackage[normalem]{ulem}

\usepackage{epsf}
\usepackage{graphicx,psfrag}

\usepackage{frcursive}
\usepackage{helvet}

\usepackage[mathscr]{euscript}

\usepackage{srcltx}

\usepackage[T1]{fontenc}
\font\sans=cmss12

 at 12truept 
\font\tsc=cmcsc8 at 9truept

\def \hvC{\text{\fontfamily{phv}\selectfont{C}}}

\def \SaK{\text{\sans K}}

\def \ScptA{\mathscr{A}}
\def \ScptB{\mathscr{B}}

\def \CalH{\mathcal H}
\def \CalI{\mathcal I}
\def \CalJ{\mathcal J}

\def \CalR{\mathcal R}

\def \Bk{\mathscr{B}}

\def \Gk{\mathscr{G}}

\def \Lk{\mathscr{L}}

\def \Nk{\mathscr{N}}

\def \Pk{\mathscr{P}}

\def \Sk{\mathscr{S}}
\def \Tk{\mathscr{T}}
\def \Uk{\mathscr{U}}

\def \mB{\text{\rm{\bf{B}}}}

\def \mG{\text{\rm{\bf{G}}}}

\def \mL{\text{\rm{\bf{L}}}}

\def \mN{\text{\rm{\bf{N}}}}

\def \mS{\text{\rm{\bf{S}}}}

\def \mU{\text{\rm{\bf{U}}}}

\def \mField{k}

\def \bZ{{\mathbb Z}}

\def \bC{{\mathbb C}}

\def \bF{{\mathbb F}}
\def \bG{{\mathbb G}}

\def \bZ{{\mathbb Z}}

\def \fka{{\mathfrak a}}

\def \fkg{{\mathfrak g}}

\def \fkl{{\mathfrak l}}

\def \frcC{ {\text{\rm{\fontsize{9}{10}{\fontfamily{frc}\selectfont{C}}}}} }

\def \myauthor{Dan Barbasch and Allen Moy}

\def \Ad{\text{\rm{Ad}}}

\def \staralg{\star}

\def \Cic{{C^{\infty}_{c}}}
\def \End{\text{\rm{End}}}
\def \GL{\text{\rm{GL}}}
\def \Hom{\text{\rm{Hom}}}
\def \mygrad{\text{\rm{grad}}}
\def \myinf{\rho}
\def \myId{\text{\rm{I}}}

\def \myprod{\star}

\def \Lie{\text{\rm{Lie}}}
\def \meas{\text{\rm{meas}}}

\def \matrixskip{\begin{matrix}\, \\ \, \end{matrix}}

\def\myhecke#1#2{ {\hskip 0.005in}_{{\hskip 0.005in}_{#1}} {\mathcal H}_{{\hskip 0.005in}_{#2}}{\hskip 0.005in} }

\def\mydhecke#1#2#3{ {\hskip 0.005in}_{{\hskip 0.005in}_{#1}} {\mathcal H}^{#3}_{{\hskip 0.005in}_{#2}}{\hskip 0.005in} }

\def\mymatrix#1#2#3#4{ {#1}^{#2}_{ #3 , #4 } }

\def\theequation{\ifnum\value{subsection}>0\relax
\thesubsection.\arabic{equation}\relax
\else\ifnum\value{section}>0\relax
\thesection.\arabic{equation}\relax \else\arabic{equation}\fi\fi}

\newtheorem{thm}[equation]{Theorem}
\newtheorem{lemma}[equation]{Lemma}
\newtheorem{prop}[equation]{Proposition}
\newtheorem{cor}[equation]{Corollary}

\newtheorem{defn}[equation]{Definition}

\newcommand\ovl{\overline}

\begin{document}
\vskip 0.40in
\title{Peter--Weyl Iwahori algebras}
\author{\myauthor}
%%%%%%%%%%%%%%%%%%%%%%%%%%%%%%%%%%%%%%%%%%%%%%%%%%%%%%%%%%%%%%%%%%%%%%%%%
%%%%%%%%%%%%%%%%%%%%%%%%%%%%%%%%%%%%%%%%%%%%%%%%%%%%%%%%%%%%%%%%%%%%%%%%%
%%%%%%%%%%%%%%%%%%%%%%%%%%%%%%%%%%%%%%%%%%%%%%%%%%%%%%%%%%%%%%%%%%%%%%%%%

\thanks{The first author is partly supported by NSA grant H98230-16-1-0006.}

\thanks{The second author is partly supported by Hong Kong Research Grants Council grant CERG {\#}16301915.}



%%%%%%%%%%%%%%%%%%%%%%%%%%%%%%%%%%%%%%%%%%%%%%%%%%%%%%%%%%%%%%%%%%%%%%%%%
%%%%%%%%%%%%%%%%%%%%%%%%%%%%%%%%%%%%%%%%%%%%%%%%%%%%%%%%%%%%%%%%%%%%%%%%%
%%%%%%%%%%%%%%%%%%%%%%%%%%%%%%%%%%%%%%%%%%%%%%%%%%%%%%%%%%%%%%%%%%%%%%%%%

\begin{abstract}

The Peter--Weyl idempotent $e_{\Pk}$ of a parahoric subgroup ${\Pk}$ is the sum of the idempotents of irreducible representations of $\Pk$ which have a nonzero Iwahori fixed vector.   The convolution algebra associated to $e_{\Pk}$ is called a Peter--Weyl Iwahori algebra.  We show any Peter--Weyl Iwahori algebra is Morita equivalent to the Iwahori-Hecke algebera.  Both the Iwahori--Hecke algebra and a Peter--Weyl Iwahori algebra have a natural $\hvC^\staralg$-algebra structure, and the Morita equivalence preserves irreducible hermitian and unitary modules.  Both algebras have another anti-involution denoted as $\bullet$, and the Morita equivalence preserves irreducible and unitary modules for the $\bullet$-involution.   

\end{abstract}

%%%%%%%%%%%%%%%%%%%%%%%%%%%%%%%%%%%%%%%%%%%%%%%%%%%%%%%%%%%%%%%%%%%%%%%%%
%%%%%%%%%%%%%%%%%%%%%%%%%%%%%%%%%%%%%%%%%%%%%%%%%%%%%%%%%%%%%%%%%%%%%%%%%
%%%%%%%%%%%%%%%%%%%%%%%%%%%%%%%%%%%%%%%%%%%%%%%%%%%%%%%%%%%%%%%%%%%%%%%%%
\maketitle
%%%%%%%%%%%%%%%%%%%%%%%%%%%%%%%%%%%%%%%%%%%%%%%%%%%%%%%%%%%%%%%%%%%%%%%%%
%%%%%%%%%%%%%%%%%%%%%%%%%%%%%%%%%%%%%%%%%%%%%%%%%%%%%%%%%%%%%%%%%%%%%%%%%
%%%%%%%%%%%%%%%%%%%%%%%%%%%%%%%%%%%%%%%%%%%%%%%%%%%%%%%%%%%%%%%%%%%%%%%%%

\section{Introduction}

\bigskip

Let $\mField$ be a non-archimedean local field with ring of integers ${\CalR}_{\mField}$ and prime ideal $\wp_{\mField}$.  Suppose $\Gk = \mG (\mField )$ is the group of $\mField$-rational points of a split reductive group defined over $\mField$ (for convenience we also assume simple).  After the choice of a Haar measure on $\Gk$, the vector space $\Cic (\Gk )$ of locally constant compactly supported functions is a convolution algebra, and any smooth representation $(\pi, V_{\pi})$ of $\Gk$ integrates to a representation of $\Cic (\Gk )$.  The algebra $\Cic (\Gk )$ has the structure of a $\hvC^\staralg$-algebra with the $\staralg$-operation given by {\,}$f^{\staralg}(g) \, = \, {\overline{f(g^{-1})}}$.  The $\hvC^\staralg$-algebra $\Cic (\Gk )$ is used to transfer problems of analysis on the group $\Gk$ to algebraic problems. In particular, we are interested in the Bernstein component of (smooth irreducible) representations with nonzero Iwahori fixed vectors.   In this setting, we fix an Iwahori subgroup ${\CalI}$ and replace $C^{\infty}_{c}(\Gk)$ by the Iwahori-Hecke algebra ${\CalH} \, := \, {\CalH} (\Gk , {\CalI})$  of ${\CalI}-$bi-invariant locally constant compactly supported functions.  The Iwahori-Hecke algebra inherits a star operation from $\Cic (\Gk)$.  \

\begin{defn}
A smooth representation $(\pi,V_{\pi})$ of $\Gk$ admits an invariant
hermitian form if there is a nontrivial hermitian form $\langle\
,\ \rangle$ satisfying
$$
\langle\pi(x)v_1,\pi(x)v_2\rangle =\langle v_1,v_2\rangle.
$$    
For $f \in \Cic (\Gk)$, we get $\langle \pi(f)v_1 , v_2 \rangle = \langle v_1 , \pi(f^{\staralg})v_2\rangle$.  If $(\pi,V_{\pi})$ is irreducible, the form is unique  up to a nonzero scalar.  The representation $(\pi,V_{\pi})$ is said to be unitarizable if $V_{\pi}$ admits a positive definite invariant hermitian form. 
\end{defn}

The  $\staralg-$operation on the Iwahori--Hecke algebra can be given explicitly in terms of the generators and relations of ${\CalH}$ (see section 5 \cite{BM2}).  The algebra $\CalH$ has another anti--involution operation (see \cite{BC}) denoted by $\bullet$, that is only given in terms of the generators and relations.  One can also definite hermitian and unitary modules for the $\bullet$-involution.  

\medskip

The graded Hecke algebra possesses analogous $\staralg$ and $\bullet$
anti--involutions.   The $\staralg$ involution is defined in terms of
generators and relations in \cite{BM2}, while both involutions are
treated from a different point of view in \cite{Op}.

\medskip

Suppose ${\Pk}$ is a parahoric subgroup containing our chosen Iwahori subgroup ${\CalI}$.  Set $\Xi$ to be the set of irreducible representations of ${\Pk}$ which contain the trivial representation of ${\CalI}$.  We define the Peter--Weyl idempotent to be the idempotent 
$$
e_{\Pk} \ := \ \ {\frac{1}{\meas ({\Pk})}} \ \ \ {\underset {\sigma \in \Xi \, \ } \sum}  \ \deg (\sigma) \ \Theta_{\sigma} \ ,
$$
and we define the Peter-Weyl Iwahori algebra as 
$$
{\CalH}( \Gk , {e_{\Pk}}) \ := \ e_{\Pk} \ \myprod \ \Cic (\Gk ) \ \myprod \ e_{\Pk} \ \ .
$$
\noindent{When} ${\Pk}$ equals ${\CalI}$, the Peter--Weyl Iwahori algebra ${\CalH}( \Gk , {e_{\Pk}})$ is the Iwahori-Hecke algebra.  For any $\Pk \supset \CalI$, it is known (see Proposition \ref{key-prop}) $e_{\Pk} \myprod e_{\CalI} = e_{\CalI} = e_{\CalI} \myprod e_{\Pk}$; consequently ${\CalH}( \Gk , e_{\CalI} ) \subset {\CalH}( \Gk , e_{\Pk} )$.  

\medskip

The Peter-Weyl Iwahori algebra ${\CalH}( \Gk , {e_{\Pk}})$ inherits a $\hvC^\staralg$-structure from $\Cic (\Gk )$.   The problem we are concerned with in this paper, is to show:
\smallskip
\begin{itemize}
\item[(i)] \ Each Peter-Weyl Iwahori algebra is Morita equivalent to the 
Iwahori--Hecke algebra, and furthermore, the Morita equivalence
preserves $\staralg$-hermitian and unitary modules.  The equivalence
is established by showing the idempotent $e_{\CalI} \, \in \,
{\CalH}(\Gk, e_{\Pk})$ is a full idempotent, i.e., $e_{\Pk} \, \in \,
{\CalH}(\Gk, e_{\Pk}) \myprod e_{\CalI} \myprod {\CalH}(\Gk,
e_{\Pk})$.  In fact, $e_{\CalI}$ belongs to ${\CalH}(\Pk, e_{\Pk})$, and is already a full idempotent, i.e., $e_{\Pk} \, \in \, {\CalH}(\Pk, e_{\Pk}) \myprod e_{\CalI} \myprod {\CalH}(\Pk, e_{\Pk})$
\smallskip
\item[(ii)] \ Each Peter--Weyl Iwahori algebra possesses an anti--involution {\,}$\bullet$, which restricts to the $\bullet$ involution on the Iwahori-Hecke algebra, and the Morita equivalence preserves {\,}$\bullet$-hermitian and unitary modules.
\end{itemize}

\vskip 0.50in

%%%%%%%%%%%%%%%%%%%%%%%%%%%%%%%%%%%%%%%%%%%%%%%%%%%%%%%%%%%%%%%%%%%%%%%%%
%%%%%%%%%%%%%%%%%%%%%%%%%%%%%%%%%%%%%%%%%%%%%%%%%%%%%%%%%%%%%%%%%%%%%%%%%
%%%%%%%%%%%%%%%%%%%%%%%%%%%%%%%%%%%%%%%%%%%%%%%%%%%%%%%%%%%%%%%%%%%%%%%%%

\section{Preliminaries on the Iwahori--Hecke algebra}

\subsection{\ }

We introduce notation:
 
\begin{itemize}

\item[$\bullet$] \ \ $\mField \supset {\CalR}_{\mField} \supset \wp_{\mField}$, $\mG$, $\Gk = \mG (\mField )$ are as in the introduction.  Set $\bF = {\CalR}_{\mField}/\wp_{\mField}$, and let $q$ denote the order of $\bF$. 

\item[$\bullet$] \ \ For any $\mField$-subgroup $\mL \subset \mG$, let $\Lk = \mL (\mField )$ denote the group of $\mField$-rational points.  Denote by $\fkg$, and $\fkl$ the obvious Lie subalgebras of $(\Lie(\mG))(\mField)$.

\item[$\bullet$] \ \ Let $\mS \subset \mG$ denote a maximal split
  $\mField$-torus (which we can in fact assume defined over
  ${\CalR}$), and $\Sk = \mS (\mField )$; so, the characters $Y = \Hom
  ( \mS, \mG_{m})$ are  naturally paired with the cocharacters $X = \Hom (\mG_{m} , \mS)$.  Set $\Sk^{0} = \mS (\CalR )$, and denote by $\Sk^{+}$ the maximal pro-p-subgroup of $\Sk^{0}$.

\item[$\bullet$] \ \ $R$ is the set of roots of $\mG$ with respect to $\mS$.  To a root $\alpha$, we denote by $\mU_{\alpha} \subset \mG$ and $\Uk_{\alpha} = \mU_{\alpha}(\mField) \subset \Gk$, the corresponding root groups.  For a choice  $R^{+} \subset R$ of positive roots $R^{+}$, let $\mB$ (and $\Bk = \mB (\mField)$) be the associated Borel subgroups, and $\Pi \subset R^{+}$, the simple roots.  

The choice of a Chevalley basis of $\fkg$ allows us to define $\mG$ and $\mU_{\alpha}$ over the integers ${\CalR}_{\mField}$, and thus over $\bF$ too (we write $\mG \times_{{\CalR}_{{\mField}}} \bF$ for the group over $\bF$), so that  $\mG \times_{{\CalR}_{\mField}} \bF$ is a connected reductive split $\bF$-group, and there is a canonical identification of the root systems of $\mG$ and $\mG \times_{{\CalR}_{\mField}} \bF$.

\item[$\bullet$] \ \ Let $\ScptB$ denote the Bruhat-Tits building of $\Gk$.   The torus $\mS$ (defined over ${\CalR}_{\mField}$) yields an apartment $\ScptA \subset \ScptB$, and $\ScptB$ is the union of all its apartments. The Chevalley basis above allows us to:

\smallskip
\begin{itemize}
\item[(i)] embed $Y$ inside $\ScptA$, so that the origin $0$ becomes a hyperspecial point
\smallskip
\item[(ii)] define the set of affine roots $\Psi \, = \, \{ \ \alpha \, + \, j \ | \ \alpha \, \in \, R \ , \ j \, \in \, \bZ \ \}$ on $\ScptA$ and to each affine root $\psi$ an affine root groups $\Uk_{\psi} \subset \Uk_{\mygrad (\psi )}$.  
\end{itemize}

\smallskip

\noindent{The} assumption that $\mG$ is split simple means $\ScptB$ is a simplicial complex.  For any facet $E \subset \ScptB$, let $\Gk_{E}$ be the associated parahoric subgroup.  When $E \subset \ScptA$:
\end{itemize}
$$
\aligned
\Gk_{E} \ = \ &{\text{\rm{subgroup of {\,}$\Gk${\,} generated by {\,}$\Sk^{0}${\,} and ${\,}\Uk_{\psi}$}}} \ \ {\text{\rm{($\psi$ satisfying {\,}$\psi (x) \ge 0${\,} for all $x \in E$)}}} \\
\Gk^{+}_{E} \ = \ &{\text{\rm{subgroup of {\,}$\Gk${\,} generated by {\,}$\Sk^{+}${,} and ${\,}\Uk_{\psi}{\,}$}}} \ \ {\text{\rm{($\psi$ satisfying {\,}$\psi (x) > 0${\,} for all $x \in E$)}}}
\endaligned
$$

\begin{itemize}
\item[$\bullet$] \ Fix a Haar measure on $\Gk$, and therefore a convolution product {\,}$\myprod${\,} on $\Cic (\Gk)$.    For any open compact subgroup $J \subset \Gk$, let $1_{J}$ denote the characteristic function, and set

\begin{equation}\label{groupidem}
e_{J} \ := \ {\frac{1}{\meas ( J )}} \ 1_{J} \ .
\end{equation}

\noindent{When} a facet $E \subset \ScptB$ is of maximal dimension, i.e., $E$ is a chamber, the parahoric subgroup $\CalJ = \Gk_{E}$ is an Iwahori subgroup.  Set
\smallskip
$$
\aligned
{\hskip 0.50in}{\CalH}({\Gk},{\CalJ}) \ :&= \ e_{\CalJ} \, \myprod \, \Cic (\Gk ) \, \myprod \, e_{\CalJ} \quad {\text{\rm{(Iwahori--Hecke algebra with respect to {\,}${\CalJ}$).}}} 
\endaligned
$$
\smallskip
\noindent{We} recall any two chambers of $\ScptB$ belong to the same $\Gk$ orbit; so any two Iwahori subgroups are conjugate in $\Gk$.

\smallskip

The choices of a set of positive roots $R^{+}$ and a Chevalley basis singles out a particular Iwahori subgroup $\CalI$ which can be described as follows: \  To the facet $\{ \, 0 \, \} \subset \ScptA$, and its maximal parahoric subgroup $\Gk_{\{ 0 \}}$, we consider the quotient map $\Gk_{\{ 0 \}} \longrightarrow (\Gk_{\{ 0 \}} / \Gk^{+}_{\{ 0 \}}) = (\mG \times_{{\CalR}_{\mField}} \bF) (\bF)$.  Then, ${\CalI}$ is the inverse image of the Borel subgroup of $(\mG \times_{{\CalR}_{\mField}} \bF) (\bF)$ corresponding to the positive roots $R^{+}$.  

\smallskip

We recall that the Iwahori--Hecke algebra ${\CalH}(\Gk, \CalI )$ has a presentation in terms of the finite Iwahori--Hecke agebra ${\CalH}(\Gk_{\{ 0 \}} , \CalI )$ and $X$ (which can viewed as a rational functions on the torus $Y \otimes_{\bZ} \bC^{\times}$):

\smallskip

\begin{itemize}
\item[(i)] \ Let $\mN$ (defined over ${\CalR}_{\mField}$) be the normalizer of $\mS$.  For each $n \in \mN (\CalR_{\mField} ) \subset \Gk_{\{ 0 \}}$, set 
$$
T_{n} \ = \ {\frac{1}{\meas (\CalI )}} \ 1_{\CalI \, n \, \CalI }  \ \ \in \ \Cic(\Gk) \qquad  {\text{\rm{(depends only on the coset {\,}$n \, \Sk^{0}$).}}}
$$

\smallskip

If $n_1, \, n_2 \, \in \mN (\CalR_{\mField} )$ have the length property that $\ell(n_1 n_2) \, = \, \ell (n_1) \, + \, \ell (n_2)$, then $T_{n_1} \myprod T_{n_2} \, = \, T_{n_1 n_2}$.   

\medskip

For each simple root $\alpha$, let $t_{\alpha} \in \mN (\CalR_{\mField} )$ be an element whose action on $X$ is the reflection $s_{\alpha}$, and set 
\smallskip
$$
T_{s_\alpha} \ = \ {\frac{1}{\meas (\CalI )}} 1_{\CalI \, t_{\alpha} \, \CalI }  \ .
$$
\smallskip
\noindent{Then}, $T^{2}_{s_{\alpha}} \ = \ (q-1) \, T_{s_{\alpha}} \ + \ q \, \myId$.

\medskip

\item[(ii)] There is an embedding, due to Bernstein (see {\S}3 \cite{Lz1}, {\S}4 \cite{Lz2}), 
$$
\aligned X \ &\longrightarrow \ {\CalH} (\Gk , \CalI ) \\
x \ &\longrightarrow \ \quad \Theta_{x} 
\endaligned
$$
\noindent{satisfying}
$$
\Theta_{x} \ T_{s_{\alpha}} \ = \ T_{s_{\alpha}} \ \Theta_{s_{\alpha}(x)} \ + \ (q-1) \ {\frac{\Theta_{x} \ - \ \Theta_{s_{\alpha}(x)}}{1 \ - \ \Theta_{-\alpha}}} \ .
$$ 

\medskip

\item[(iii)] The set of elements $\{ \ \Theta_{x} \, T_{n} \ | \ x \in X \ , \ n \in  \mN (\CalR_{\mField} )/\Sk^{0} \ \}$ is a (complex) basis of ${\CalH} (\Gk , \CalI )$.
\end{itemize}

\medskip

\item[$\bullet$]  \ The space of functions {\,}$\Cic (\Gk )${\,}, admits a natural anti--involution $\staralg$ given by 
\begin{equation}\label{cstar}
  f^{\staralg}(g):=\overline{f(g^{-1})} \ .
\end{equation}
That $\CalI$ is a subgroup means the involution restricts to an anti--involution of {\,}${\CalH}(\Gk ,\CalI )$.
\smallskip
\item[$\bullet$] \ The algebra ${\CalH} (\Gk , \CalI )$ has another anti-involution $\bullet$ (see \cite{BC}) defined in terms of the generators $T_{n}$ ($n \in \mN (\CalR_{\mField} )$), and $\Theta_{x}$ ($x \in X$), given by
\begin{equation}\label{BCinvolution}
T^{\bullet}_{n} \ = \ T_{n^{-1}} \ \ {\text{\rm{for \ $n \in \mN (\CalR_{\mField} )$ \ \ \ and \ \ \ }}} \Theta^{\bullet}_{x} \ = \ \Theta_{x} \ \ {\text{\rm{for \ $x \in X$}}} \ .  
\end{equation}

\smallskip

\item[$\bullet$] \ Let $n_{0} \in \mN(\CalR_{\mField} )$ be a representative for the longest element in the Weyl group $\mN(\mField )/\Sk$.  There is an involution {\,}$\fka${\,} of the group $\Gk$ so that $\Sk$, $\Nk = \mN (\mField)$, and $\CalI$ are {\,}$\fka$-invariant, and:
$$
\aligned
{\fka}(x) \ &= \ n_{0} \, x^{-1} \, n^{-1}_{0} \qquad \forall \ \ x \, \in \ \Sk \\
{\fka}(n) \ &= \ n_{0} \, n \, n^{-1}_{0}  \mod \Sk \qquad \forall \ \ n \, \in \ \Nk \ .
\endaligned
$$
For the case of the group $\GL (n)$, and the standard representation realization of classical groups the involution $\fka$ is
$$
\fka \, ( \, g \, ) \ = \ n_{0} \, (g^{-1})^{\text{\rm{T}}} \, n^{-1}_{0} \ ,
$$
\noindent where ${\text{\rm{T}}}$ is transpose, and $n_{0} \in \mN (\CalR_{\mField} )$ is a monomial matrix representative of the longest Weyl element.

\medskip

The involution $\fka$, defines an involution of the Iwahori--Hecke algebra 
 ${\CalH} (\Gk , \CalI )$, and the relationship between the two anti-involutions $\staralg$ and $\bullet$ is:  
$$
\bullet \, (y) \ = \ T^{-1}_{n_{0}} \ \fka ( \, y^{\staralg} \, ) \ T_{n_{0}} \qquad \forall \ \ y \ \in \ {\CalH} (\Gk , \CalI ) \ .
$$

\end{itemize}

\vskip 0.30in

%%%%%%%%%%%%%%%%%%%%%%%%%%%%%%%%%%%%%%%%%%%%%%%%%%%%%%%%%%%%%%%%%%%%%%%%%
%%%%%%%%%%%%%%%%%%%%%%%%%%%%%%%%%%%%%%%%%%%%%%%%%%%%%%%%%%%%%%%%%%%%%%%%%
%%%%%%%%%%%%%%%%%%%%%%%%%%%%%%%%%%%%%%%%%%%%%%%%%%%%%%%%%%%%%%%%%%%%%%%%%
\section{Idempotents}

\smallskip

Let $\CalJ$ be an Iwahori subgroup.  The collection of pairs consisting of a minimal $\mField$-Levi subgroup of $\Gk$, i.e., a maximal split torus $\Tk$, and the unramified characters of $\Tk$, is the cuspidal data to parametrize a full subcategory $\Omega$ (Bernstein component) of the smooth representations (of $\Gk$).  Equivalently, $\Omega$ is the full subcategory of representations generated by their $\CalJ$ fixed vectors.  Furthermore, there is an essentially compact distribution $P_{\Omega}$ (representable by a locally $L^1$-function on the regular compact elements of $\Gk$), so that for any smooth representation $(\pi , V_{\pi})$, the endomorphism $\pi (P_{\Omega} ) \in \End_{\Gk}(V_{\pi})$ is an idempotent which projects to the $\Omega$ component of $V_{\pi}$. 
 
\medskip

Suppose $F \subset \ScptB$ is a chamber (so $\Gk_{F}$ is an Iwahori subgroup), and $E$ is a facet in $F$ (so $\Gk_{E} \supset \Gk_{F} \supset \Gk^{+}_{F} \supset \Gk^{+}_{E}$).  The group $\Gk_{F}/\Gk^{+}_{E}$ is a Borel subgroup of $\Gk_{E}/\Gk^{+}_{E}$.  Let $\myinf$ be the quotient map from $\Gk_{E}$ to $\Gk_{E} /\Gk^{+}_{E}$.  We define the Peter--Weyl idempotent associated to the facet {\,}$E${\,} as (see also \cite{BCM})
\begin{equation}\label{pw-idempotent}
e_{E} \ := \ {\frac{1}{\meas (\Gk_{E} )}} \ \Big( \ {\underset {\sigma \in \Xi}  \sum}  \ \deg (\sigma) \ \Theta_{\sigma} \, \circ \, \myinf \ \Big) \ , 
\end{equation}
\noindent{where} {\,}$\Xi${\,} is the collection of irreducible representations 
of $\Gk_{E}/\Gk^{+}_{E}$ which contain a nonzero (Borel) $\Gk_{F}/\Gk^{+}_{E}$-fixed vector, and $\Theta_{\sigma}$ is the character of $\sigma$.  We remark that: 

\smallskip

\begin{itemize} 
\item[(i)] \ If $F'$ is another chamber containing $E$, then $\Gk_{F'}/\Gk^{+}_{E}$ is a Borel subgroup of $\Gk_{E}/\Gk^{+}_{E}$ too, and the right side of the above definition yields the same idempotent $e_{E}$.
\smallskip
\item[(ii)]  \ For a chamber $F$, the idempotent $e_{F}$ equals $e_{\Gk_{F}}$. In this situation, we will use both notations.
\end{itemize} 

\medskip
\begin{defn}
The Peter--Weyl Iwahori algebra associated to the idempotent
{\,}$e_{E}$ is the algebra
\smallskip
\begin{equation}\label{pw-algebra}
{\CalH}_{E} \ := \ e_{E} {\,}\myprod{\,} \Cic (\Gk) {\,}\myprod{\,} e_{E} \ .
\end{equation}

\smallskip
\noindent{For} convenience, we sometimes abbreviate the name to Peter--Weyl algebra. When $F$ is a chamber, then ${\CalH}_{F}${\,} is the Iwahori--Hecke algebra {\,}${\CalH}(\Gk , \Gk_{F})$.  
  
\end{defn}

\medskip

\begin{prop}\label{idempotent-a} Let $P_{\Omega}$ be the projector for the unramified Bernstein component $\Omega$.  Suppose $E \subset \ScptB$ is a facet.  Then,  
$$
e_{E} \ = \ P_{\Omega} \ \myprod \ {\frac{1}{\meas (\Gk^{+}_{E} )}} \ 1_{\Gk^{+}_{E}} \ .
$$
\end{prop}

\vskip 0.20in

To prove Proposition {\ref{idempotent-a}}, we first establish the following lemma.

\smallskip

\begin{lemma}\label{idempotent-b}
  If $\phi \, , \, \psi \ \in \ \Cic (\Gk )$ satisfy $\pi(\phi)=\pi(\psi)$ for all (irreducible smooth) $\pi,$ then $\phi \, = \, \psi.$
\end{lemma}
\begin{proof}
Let $d \mu$ denote the Plancherel measure on ${\widehat{\Gk}_{\text{temp}}}$.  By the Plancherel formula,

$$
\forall \ \ F \ \in C^{\infty}_{c} (\Gk) \quad : \quad F(1) \ = \ \int_{\widehat{\Gk}_{\text{temp}}} {\text{trace}} \, ( \, \pi (F) \, ) \ d \mu (\pi) \ .
$$

\noindent{For} $\phi \in C^{\infty}_{c}(\Gk)$ and $x \in \Gk$, set $F_{x}(\phi) \, = \, \delta_{x} \star \phi$, so $F_{x}(\phi) \, (g) \, = \, \phi (x^{-1}g)$.  If $\phi \, , \, \psi \in C^{\infty}_{c}(\Gk)$, and $\pi ( \phi ) = \ \pi (\psi )$ for all $\pi$, then, for all $x \in \Gk$, we have $\pi (F_{x}(\phi ) ) \, = \, \pi ( \delta_{x} \star \phi) \, = \, \pi ( \delta_{x} \star \psi) \, = \, \pi (F_{x}(\psi ) )$. Thus,

$$
\aligned
\phi(x^{-1})  \ = \ F_{x}(\phi) \, (1) \ &= \ \int_{\widehat{\Gk}_{\text{temp}}} {\text{trace}} \, ( \, \pi (F_{x}(\phi) \, ) \ d \mu (\pi) \ = \ \int_{\widehat{\Gk}_{\text{temp}}} {\text{trace}} \, ( \, \pi (F_{x}(\psi) \, ) \ d \mu (\pi) \\
&= \ F_{x}(\psi) \, (1) \ = \ \psi (x^{-1}) \ .
\endaligned
$$

\noindent{So,} $\phi \, = \, \psi$.

\end{proof}

\medskip

\noindent{We} note we can replace $\Gk$ by a compact group $J$, and the analogous result holds for any two $\phi \, , \, \psi \ \in \ \Cic (J )$

\bigskip

\noindent{\it{Proof of Proposition{\,}{\ref{idempotent-a}}}.} \ \ Suppose $(\pi , V_{\pi})$ is a smooth irreducible representation.   The operator  $\pi ( \, {P_{\Omega}} \, ) \, \in \, \End (V_{\pi})$ is the scalar 1 if $\pi$ has a nonzero Iwahori $\Gk_{F}$-fixed vector and the scalar 0 otherwise.  The operator $\pi ( \, P_{\Omega} \star e_{\Gk^{+}_{E}} \, )$ is projection to the subspace $V^{\Gk^{+}_{E}}_{\pi}$.   By Theorem 5.2 in \cite{MP1} and Proposition 6.2 in \cite{MP2}, if 
$\pi$ has a nonzero $\Gk_{F}$-vector, i.e., an unrefined depth zero minimal $\SaK$-type consisting of the trivial representation of $\Gk_{F}$, then any other irreducible representation of $\Gk_{E}$ in $V^{\Gk^{+}_{E}}_{\pi}$ must contain a nonzero $\Gk_{F}$-fixed vector.  Clearly (from the representation theory of finite groups) for any irreducible representation $\pi$, the operator $\pi (e_{E})$ is projection to
the subspace $V^{\Gk^{+}_{E}}_{\pi}$.  It follows $\pi ( \, P_{\Omega} \star
e_{\Gk^{+}_{E}} \, ) \, = \, \pi ( e_{E} )$, and so, by Lemma \ref{idempotent-b},
${P_{\Omega}} \star e_{\Gk^{+}_{E}} \, = \, e_{E}$.  

\hfill \qed

\vskip 0.30in

%%%%%%%%%%%%%%%%%%%%%%%%%%%%%%%%%%%%%%%%%%%%%%%%%%%%%%%%%%%%%%%%%%%%%%%%%%%%%%%
%%%%%%%%%%%%%%%%%%%%%%%%%%%%%%%%%%%%%%%%%%%%%%%%%%%%%%%%%%%%%%%%%%%%%%%%%%%%%%%
%%%%%%%%%%%%%%%%%%%%%%%%%%%%%%%%%%%%%%%%%%%%%%%%%%%%%%%%%%%%%%%%%%%%%%%%%%%%%%%
%%% TWO SIDED IDEAL

Suppose $E$ is a facet inside a chamber $F$.    We recall that $\Gk^{+}_{E} \subset \Gk^{+}_{F} \subset \Gk_{F} \subset \Gk_{E}$ and $\Gk^{+}_{E}$ is a normal subgroup of $\Gk_{E}$.  The idempotents $e_{\Gk^{+}_{E}}$, $e_{F}$ and $e_{E}$ belong to the finite dimensional algebra 
$$
{\CalH} \, := \, e_{\Gk^{+}_{E}} \myprod \Cic (\Gk_{E} ) \myprod e_{\Gk^{+}_{E}} \ \ . 
$$
\noindent This algebra is equal to the canonical embedding of $\Cic (\Gk_{E}/\Gk^{+}_{E})$ in $\Cic (\Gk_{E} )$ via the quotient map $\rho \, : \, \Gk_{E} \, \rightarrow \, \Gk_{E}/\Gk^{+}_{E}$.   It is a consequence of the normality of $\Gk^{+}_{E}$ in $\Gk_{E}$ that $\Cic (\Gk_{E}) {\,}{\myprod}{\,}  e_{F} {\,} = {\CalH} {\,}{\myprod}{\,}  e_{F} {\,}$ and $e_{F} {\,}{\myprod}{\,}  \Cic (\Gk_{E}) = e_{F} {\,}{\myprod}{\,} {\CalH}$.  If $\kappa$ is an irreducible representation of $\Gk_{E}$, with character $\Theta_{\kappa}$, set 
$$
e_{\kappa} := {\frac{1}{\meas (\Gk_{E} )}} \ \deg (\kappa) \ \Theta_{\kappa} 
$$

\smallskip

\noindent to be the (central) idempotent in $\Cic (\Gk_{E} )$ associated to $\kappa$.   Clearly, $e_{\kappa} {\,}\myprod {\,} \Cic (\Gk_{E}) {\,}{\myprod}{\,}  e_{F} {\,}  {\,}{\myprod}{\,}  \Cic (\Gk_{E})  {\,}\myprod {\,} e_{\kappa}$ is an ideal of $\Cic (\Gk_{E})$, which is either a minimal ideal or zero.  

\medskip

Define 

\begin{equation}\label{hecke-alg-finite}
\aligned
{\CalH}^{\text{\rm{fin}}}_{E} :&= e_{E} \myprod \Cic (\Gk_{E}) \ {\myprod} \ e_{E} \quad &&{,} {\hskip 0.45in} {{\CalH}^{\text{\rm{fin}}}_{F}} := e_{F} \myprod \Cic (\Gk_{E}) \ {\myprod} \ e_{F}  \\
{\mydhecke{E}{F}{\text{\rm{fin}}}} \ :&= \ e_{E} \ {\myprod} \Cic (\Gk_{E}) \ {\myprod} \ e_{F} \quad &&{,} \qquad {\mydhecke{F}{E}{\text{\rm{fin}}}} := e_{F} \ {\myprod} \ \Cic (\Gk_{E}) \ {\myprod} \ e_{E} \\
\endaligned
\end{equation}  

\medskip

\begin{prop}\label{key-prop} \ \ The (finite dimensional) vector space  
$\Cic (\Gk_{E}) {\,}{\myprod}{\,} e_{F} {\,}{\myprod}{\,}  \Cic (\Gk_{E}) $
is a bi-module for $\Cic (\Gk_{E} )$, i.e., an ideal of $\Cic (\Gk_{E} )$, and

\smallskip

\begin{itemize}

\item[(i)] \ It equals ${\CalH}^{\text{\rm{fin}}}_{E}$.

\smallskip

\item[(ii)] \ It is the span of matrix coefficients of the representations with $\Gk_{F}$-fixed vectors. 

\smallskip

\item[(iii)] \ $e_{E} \ \in \  \Cic (\Gk_{E}) {\,}{\myprod}{\,} e_{F} {\,}{\myprod}{\,} \Cic (\Gk_{E})$, i.e., $e_{F} \in \Cic (\Gk_{E}) {\,}{\myprod}{\,} e_{\Gk_{F}} {\,}{\myprod}{\,} \Cic (\Gk_{E})$ is a full idempotent.

\smallskip

\item[(iv)] \ $e_{E} {\,}{\myprod}{\,}  e_{F} \ = \ e_{F} {\,}{\myprod}{\,} e_{E} \ = \ e_{F}$.

  \smallskip

\item[(v)] \  ${\CalH}^{\text{\rm{fin}}}_{E} \ = \ {\mydhecke{E}{F}{\text{\rm{fin}}}} {\,}\myprod{\,} {\mydhecke{F}{E}{\text{\rm{fin}}}}$, and \ ${\CalH}^{\text{\rm{fin}}}_{F} \ = \ {\mydhecke{F}{E}{\text{\rm{fin}}}} {\,}\myprod{\,} {\mydhecke{E}{F}{\text{\rm{fin}}}}$.

\end{itemize}
\end{prop}

\begin{proof}  \ \  To prove statement (i), suppose $h_1 {\,}{\myprod}{\,}  e_{\Gk_{F}} {\,}{\myprod}{\,}  h_2 \, \in \, \Cic (\Gk_{E}) {\,}{\myprod}{\,}  e_{F} {\,}{\myprod}{\,}  \Cic( \Gk_{E} )$.  For any representation $\kappa$  of $\Gk_{E}$, we have $\kappa ( h_1 {\,}{\myprod}{\,} e_{F} {\,}{\myprod}{\,} h_2 ) \, = \, \kappa (h_1) \,  \kappa (e_{F}) \, \kappa (h_2)$.  When $\kappa$ is irreducible, we deduce that 
$e_{\kappa} \myprod \Cic (\Gk_{E}) {\,}{\myprod}{\,}  e_{F} {\,}{\myprod}{\,}  \Cic( \Gk_{E} ) \myprod e_{\kappa}$ is zero if and only if $ \kappa (e_{F})$ is zero. 

\smallskip

Suppose $(\sigma, V_{\sigma})$ is an irreducible representation of $\Gk_{E}/\Gk^{+}_{E}$ which contains a nonzero $\Gk_{F}/\Gk^{+}_{E}$-fixed vector, i.e., $\sigma \in \Xi$.  Set $\kappa = \sigma \circ \rho$.   Then, $\kappa (e_{F})$ is nonzero which means $e_{\kappa} \myprod \Cic (\Gk_{E} ) {\,}{\myprod}{\,}  e_{F} {\,}{\myprod}{\,} \Cic (\Gk_{E}) \myprod e_{\kappa}$ is a nonzero ideal of $\Cic (\Gk_{E} )$ which is contained in the minimal ideal  $e_{\kappa} \myprod \Cic (\Gk_{E}) \myprod e_{\kappa}$.  Therefore, 
$$
e_{\kappa} \myprod \Cic (\Gk_{E} ) {\,}{\myprod}{\,}  e_{F} {\,}{\myprod}{\,} \Cic (\Gk_{E}) \myprod e_{\kappa} \ = \ e_{\kappa} \myprod \Cic (\Gk_{E}) \myprod e_{\kappa}
$$

\noindent Since $e_{E} = {\underset {\sigma \in \Xi} \sum } e_{\sigma \circ \rho}$.  We deduce statement (i).

\medskip

For the sake of completeness, we consider when  $(\kappa, V_{\kappa})$ is an irreducible representation of $\Gk_{E}$ which does not contain a nonzero $\Gk_{F}$-fixed vector, i.e., $\kappa (e_{F})$ is zero.   Then, $\kappa (h_1 {\,}{\myprod}{\,}  e_{F} {\,}{\myprod}{\,}  h_2) = 0$.  Consequently, $\kappa (F) \, = \, 0$ for any $F \, \in \, \Cic (\Gk_{E}) \myprod e_{F} \myprod \Cic (\Gk_{E})$.  This means $e_{\kappa} \myprod \Cic (\Gk_{E} ) {\,}{\myprod}{\,}  e_{F} {\,}{\myprod}{\,} \Cic (\Gk_{E}) \myprod e_{\kappa}$ is zero.

\medskip

Statements (ii) and (iii) follow from statement (i).

\medskip

For statement (iv), the equality $e_{E} {\,}{\myprod}{\,}  e_{F} \, = \, e_{F} {\,}{\myprod}{\,} e_{E}$ follows from the fact that $e_{E}$ is a central element of $\Cic (\Gk_{E} )$.  The equality $e_{F} {\,}{\myprod}{\,} e_{E} \, = \, e_{F}$ follows from the fact that  $\kappa(e_{E} {\,}{\myprod}{\,}  e_{\Gk_{F}}) \, = \, \kappa( e_{\Gk_{F}})$ for any irreducible representation $\kappa$ of $\Gk_{E}$, and  Lemma \ref{idempotent-b}.

\medskip

Statement (v) is a consequence of statement (i).
\end{proof}

\bigskip

Suppose $F$ is a chamber in the building, so the algebra ${{\CalH}_{F}} := e_{F} \myprod \Cic (\Gk) \ {\myprod} \ e_{F}$ is an Iwahori Hecke algebra.  If $E$ is a facet of $F$, we have previously named the algebra ${\CalH}_{E} := e_{E} \myprod \Cic (\Gk) \ {\myprod} \ e_{E}$ (which contains ${\CalH}_{F}$) the Peter-Weyl Iwahori algebra (associated to $E$).

\begin{prop}\label{hecke-modules} \ Suppose $E$ is a facet inside a chamber $F$. Then,

\begin{itemize}
\item[(i)] \ \ The idempotent {\,}$e_{F} \, \in \, {\CalH}_{E}${\,} satisfies ${\CalH}_{E} \myprod e_{F} \myprod {\CalH}_{E} \, = \, {\CalH}_{E}$, i.e., it is a full idempotent.

\smallskip

\item[(ii)] \ \ Define
\begin{equation}\label{hecke-alg-pw}
\aligned
{\myhecke{E}{F}} \ :&= \ e_{E} \ {\myprod} \Cic (\Gk) \ {\myprod} \ e_{F} \quad &&{\text{\rm{and}}} \qquad {\myhecke{F}{E}} \ := \ e_{F} \ {\myprod} \ \Cic (\Gk) \ {\myprod} \ e_{E} \\
\endaligned
\end{equation}

\noindent Then,   
$$
{\CalH}_{E} \ = \ {\myhecke{E}{F}} \ {\myprod} \ {\myhecke{F}{E}} \qquad {\text{\rm{and}}} \qquad {\CalH}_{F} \ = \ {\myhecke{F}{E}} \ {\myprod} \ {\myhecke{E}{F}} \ .
$$
\end{itemize}  
\end{prop}

\begin{proof} \ \ The proof of statement (i) is based on the analagous
  fact in Proposition \ref{key-prop} for the finite dimensional
  algebra ${\CalH}^{\text{\rm{fin}}}_{E}$.   We have
  ${\CalH}^{\text{\rm{fin}}}_{E} \, = \, {\CalH}^{\text{\rm{fin}}}_{E}
  {\,}\myprod{\,} e_{F} {\,}\myprod{\,}
  {\CalH}^{\text{\rm{fin}}}_{E}$.
Since ${\CalH}^{\text{\rm{fin}}}_{E}$ contains the identity element $e_{E}$ of ${\CalH}_{E}$, we have {\,}${\CalH}_{E} {\,}\myprod{\,}{\CalH}^{\text{\rm{fin}}}_{E} \, = \, {\CalH}_{E} \, = \,
{\CalH}^{\text{\rm{fin}}}_{E} {\,}\myprod{\,}{\CalH}_{E}$.

\noindent Therefore, 
$$
{\CalH}_{E}  {\,}\myprod{\,} e_{F} {\,}\myprod{\,}  {\CalH}_{E} \ = \ 
{\CalH}_{E}  {\,}\myprod{\,} {\CalH}^{\text{\rm{fin}}}_{E} {\,}\myprod{\,} e_{F} {\,}\myprod{\,}  {\CalH}^{\text{\rm{fin}}}_{E} {\,}\myprod{\,} {\CalH}_{E} \ \supset \ {\CalH}_{E}  {\,}\myprod{\,} e_{E} {\,}\myprod{\,} {\CalH}_{E} \ = \ {\CalH}_{E} \ ,
$$
\noindent and so $e_{F}$ is a full idempotent of ${\CalH}_{E}$ too.

\medskip

Statement (ii) can be obtained from statement (i) as follows:
  
$$
\aligned
    {\CalH}_{E} \ &= \ {\CalH}_{E} {\,}\myprod{\,} e_{F} {\,}\myprod{\,} {\CalH}_{E}  \ = \ {\CalH}_{E} {\,}\myprod{\,} e_{F} {\,}\myprod{\,} e_{F} {\,}\myprod{\,} {\CalH}_{E} \\
    &= \ {\myhecke{E}{F}} {\,}\myprod{\,} {\myhecke{E}{F}} \qquad {\text{\rm{(since ${\myhecke{E}{F}} \, = \, {\CalH}_{E} {\,}\myprod{\,} e_{F}$ and ${\myhecke{F}{E}} \, = \, e_{F}{\,}\myprod{\,} {\CalH}_{E}$).}}}
\endaligned    
$$

\smallskip

\noindent Similarly, ${\CalH}_{F} \ = \ e_{F} {\,}\myprod{\,} {\CalH}_{E} {\,}\myprod{\,} e_{F}  \ = \  e_{F} {\,}\myprod{\,} {\CalH}_{E} {\,}\myprod{\,} {\CalH}_{E} {\,}\myprod{\,} e_{F} \ = \ {\myhecke{F}{E}} {\,}\myprod{\,}  {\myhecke{E}{F}}$

\end{proof}

\vskip 0.30in

\begin{prop}\label{module-generators} \ Suppose $F$ is a chamber of $\ScptB$, and $E$ is a facet contained in $F$. 
\smallskip
\begin{itemize}
\item[(i)] \ The left ${\CalH}_{E}$-module ${\myhecke{E}{F}}$ is cyclic with generator $e_{F}$.  Similarly, the right ${\CalH}_{E}$-module ${\myhecke{F}{E}}$ is cyclic  with generator $e_{F}$
\smallskip
\item[(ii)] \ The left  $\CalH_{F}$-module  ${\myhecke{F}{E}}$  is finitely generated. Similarly, the right $\CalH_{F}$-module ${\myhecke{E}{F}}$  is finitely generated.
\end{itemize}
\end{prop}

\begin{proof} \ The first assertion of statement (i) follows from the fact that $e_{E} {\,} {\myprod} {\,} e_{F} \ = \ e_{F}$, while the second follows from $e_{F}{\myprod} {\,} e_{E} \ = \ e_{F}$.

\smallskip

To see statement (ii), we use the fact that for any two open compact subgroups $J$ and $J'$ of $\Gk$,  the space $e_{J} {\,}{\myprod}{\,} \Cic (\Gk ) {\,}{\myprod}{\,} e_{J'}$ is a finitely generated $e_{J} {\,}{\myprod}{\,} \Cic (\Gk ) {\,}{\myprod}{\,} e_{J}$ module.  We take $J \, = \, \Gk_{F}$ and $J' \, = \, \Gk^{+}_{E}$ to see $e_{\Gk_{F}} {\,}{\myprod}{\,} \Cic (\Gk ) {\,}{\myprod}{\,} e_{\Gk^{+}_{E}}$ \ ($e_{\Gk_{F}}$ is $e_{F}$) is a finitely generated ${\CalH}_{F}$-module.  If we convolve on the right by $e_{E}$, we deduce ${\myhecke{F}{E}}$ is a finitely generated 
left  $\CalH_{F}$-module.  

\smallskip

Similar reasoning shows ${\myhecke{E}{F}}$ is a finitely generated right  $\CalH_{F}-$module.

\end{proof}

\vskip 0.30in

%%%%%%%%%%%%%%%%%%%%%%%%%%%%%%%%%%%%%%%%%%%%%%%%%%%%%%%%%%%%%%%%%%%%%%%%%%%%%%%
%%%%%%%%%%%%%%%%%%%%%%%%%%%%%%%%%%%%%%%%%%%%%%%%%%%%%%%%%%%%%%%%%%%%%%%%%%%%%%%
%%%%%%%%%%%%%%%%%%%%%%%%%%%%%%%%%%%%%%%%%%%%%%%%%%%%%%%%%%%%%%%%%%%%%%%%%%%%%%%

%%%%%%%%%%%%%%%%%%%%%%%%%%%%%%%%%%%%%%%%%%%%%%%%%%%%%%%%%%%%%%%%%%%%%%%%%
%%%%%%%%%%%%%%%%%%%%%%%%%%%%%%%%%%%%%%%%%%%%%%%%%%%%%%%%%%%%%%%%%%%%%%%%%
%%%%%%%%%%%%%%%%%%%%%%%%%%%%%%%%%%%%%%%%%%%%%%%%%%%%%%%%%%%%%%%%%%%%%%%%%

\section{Matrix Coefficients} 

\medskip

We fix a Haar measgure on $\Gk$ and make $\Cic (\Gk )$ into a convolution algebra.  The map $\staralg$ defined in \eqref{cstar} is an anti--involution.  For any facet $E$, the idempotent $e_{E}$ is fixed by $\staralg$, i.e., $e^{\staralg}_{E} = e_{E}$, and therefore $\staralg$ is an anti-involution of the Peter--Weyl algebra ${\CalH}_{E}$.  When $F$ is a chamber, the $\staralg$ anti-involution on the Iwahori--Hecke algebra ${\CalH}_{F}$ is the one mentioned in the introduction.

\medskip

In the Iwahori--Hecke algebra situation, using generators, \cite{BC} defined an anti-involution $\bullet$.  In this section we show that given a facet $E$ and a chamber $F$ containing $E$ (so ${\CalH}_{F} \ \subset \ {\CalH}_{E}$), we can define an anti--involution which we also denote as $\bullet$.   The involution depends on the chamber chosen (equivalently the Iwahori subgroup inside $\Gk_{E}$); but, since any two Iwahori subgroups are conjugate, the difference is up to a conjugation.  To define $\bullet$,  we need to exhibit a decomposition of ${\CalH}_{E}$ in terms of ${\CalH}_{F}$.   In the next section we show the Morita equivalence between the Iwahori--Hecke algebra ${\CalH}_{F}$ and the Peter--Weyl Iwahori algebra ${\CalH}_{E}$ preserves $\staralg$--hermitian and $\staralg$--unitary representations as well as $\bullet$--hermitian and 
$\bullet$--unitary representations.

\medskip

\subsection{Preliminaries}

\medskip

\begin{lemma}\label{l:matrel}
Suppose $(\sigma,V_{\sigma})$ and $(\tau,V_{\tau})$ are irreducible representations of a compact group $K$, with invariant positive definite forms $\langle \, , \, \rangle_{\sigma}$ and $\langle \, , \, \rangle_{\tau}$.   If the two representations are equivalent, we assume they are equal (and abbreviate the inner product to $\langle \, , \, \rangle$). \ Suppose $x_1,x_2 \, \in \, V_{\sigma}$ and $y_1,y_2 \, \in \, V_{\tau}.$ Then  
$$
\int_{K} \langle \, x_1 \, , \, \sigma(h) x_2 \, \rangle_{\sigma} \ \overline{\langle \, y_1 \, , \, \tau (h) y_2 \, \rangle_{\tau}} \ dh \ = \
\begin{cases}
{\hskip 0.05in} 0 \qquad {\hskip 0.60in} {\text{\rm{if $\sigma$ is not equivalent to $\tau$, }}} \\
\ \\
{\hskip 0.05in} {\frac{\meas (K)}{\deg (\sigma )}} \ \langle \, x_1 \, , \, y_1 \, \rangle\ \ovl{\langle \, x_2 \, , \, y_2 \, \rangle}  \qquad {\text{\rm{if}}} \ \ \sigma = \tau \, . 
\end{cases}
$$
\end{lemma}
\begin{proof} These are the Schur orthogonality relations. The case when the representations are inequivalent is  clear.  When they are equivalent, denote them both by $\sigma$.  The tensor product $V_{\sigma} \otimes V_{\sigma}$ has two $(K \times K)$-invariant form given by \ $\langle x_1\otimes y_1 , x_2\otimes y_2 \rangle_{\sigma \otimes \sigma} \ = \ \langle x_1 , y_1 \rangle \ {\overline{\langle x_2 , y_2 \rangle}}$ \ and \ $\langle\langle x_1\otimes y_1 , x_2\otimes y_2 \rangle\rangle \ := \ \int_{K} {\langle
  x_1 , \sigma(h) x_2\rangle} \ \overline {\langle y_1 , \sigma (h) y_2\rangle} \ dh$, and so they must be scalar multiples of one another.  Evaluation of the scalar yields 
$$
\langle \langle \ , \ \rangle \rangle \ = \ {\frac{\meas (K)}{\deg (\sigma )}} \ \langle \ , \ \rangle_{\sigma \otimes \sigma} \ .
$$
\end{proof}

When $(\sigma , V_{\sigma})$ is an irreducible representation of $K$, and $u, v \, \in \, V_{\sigma}$, we define the matrix coefficient  $\mymatrix{m}{\sigma}{u}{v}$ as:
\begin{equation}\label{defn-coeff}
\mymatrix{m}{\sigma}{u}{v} (k) \ := \ \langle \, u \, , \, \sigma (k) \, v \, \rangle \ . 
\end{equation}

\medskip

\begin{cor} \label{c:conv}
\begin{itemize}
\item[(i)] \ \ With the same notation as in Lemma \ref{l:matrel}, 
\begin{equation}\label{orthogonality-relations}
\mymatrix{m}{\sigma}{x_1}{x_2} \, \myprod \, \mymatrix{m}{\tau}{y_1}{y_2} \ = \ 
\begin{cases}
\ \ 0  &{\text{\rm{if $\sigma$ is not equivalent to $\tau$}}}   \\ 
\ \ {\frac{\meas (K)}{\deg (\sigma )}} \ \ {\overline{ \langle \, x_2 \, , \, y_1 \, \rangle}} \ \ \mymatrix{m}{\sigma}{x_1}{y_2} &{\text{\rm{if  \ $\sigma \, = \, \tau$}}}  \\
\end{cases}
\end{equation}
\item[(ii)] \ \ If $v \, \in \, V_\sigma$ satisfies 
\begin{equation}\label{normalization}
\langle v , v \rangle = {\frac{\deg (\sigma )}{\meas (\Gk_{E})}} \ ,
\end{equation} 
\noindent{then} the function {\,}$\mymatrix{m}{\sigma}{v}{v}${\,} is a convolution idempotent. 
\end{itemize}
\end{cor}

\begin{proof} Obvious.
\end{proof}

\medskip

\noindent{Recall} that
$(\mymatrix{m}{\sigma}{x_1}{x_2})^{\staralg}(k) \, = \, 
\overline{\langle \, x_1 \, , \, \sigma(k^{-1}) x_2 \, \rangle} \, = \, 
\overline{\langle \, \sigma (k) x_1 \, , \, x_2\,  \rangle} \, = \, 
\langle \, x_2 \, , \, \sigma (k) x_1 \,  \rangle \, = \, \mymatrix{m}{\sigma}{x_2}{x_1}$.  Thus, 
\begin{equation}
  \label{eq:convstar}
  (\mymatrix{m}{\sigma}{x_1}{x_2})^{\staralg} \ {\myprod} \ \mymatrix{m}{\sigma}{y_1}{y_2} \ = \ \mymatrix{m}{\sigma}{x_2}{x_1} \ {\myprod} \ \mymatrix{m}{\sigma}{y_1}{y_2} \ = \ {\frac{\meas (K)}{\deg (\sigma )}} \ \
  {\overline{\langle \, x_1 \, , \, y_1 \, \rangle}} \ \ \mymatrix{m}{\sigma}{x_2}{y_2} \ .  
\end{equation}

\smallskip

\noindent{We} also note that if $\lambda_{h}$ (resp.~$\rho_{h}$) is the left (resp.~right) translation representation, i.e., $(\lambda_{h}(f)) \, (k) \, = \, f (h^{-1}k)$ and $(\rho_{h}(f)) \, (k) \, = \, f (kh)$, then 
$$
\lambda_{h}( \mymatrix{m}{\sigma}{u}{v} ) \ = \ \mymatrix{m}{\sigma}{\sigma(h) u}{v}  {\text{\rm{\qquad and \qquad}}}   \rho_{h}( \mymatrix{m}{\sigma}{u}{v} ) \ = \ \mymatrix{m}{\sigma}{u}{\sigma(h) v}  \ .
$$
\vskip 0.50in

%%%%%%%%%%%%%%%%%%%%%%%%%%%%%%%%%%%%%%%%%%%%%%%%%%%%%%%%%%%%%%%%%%%%%%%%%%%%%
%%%%%%%%%%%%%%%%%%%%%%%%%%%%%%%%%%%%%%%%%%%%%%%%%%%%%%%%%%%%%%%%%%%%%%%%%%%%%
%%%%%%%%%%%%%%%%%%%%%%%%%%%%%%%%%%%%%%%%%%%%%%%%%%%%%%%%%%%%%%%%%%%%%%%%%%%%%
%%%%%%%%%%%%%%%%%%%%%%%%%%%%%%%%%%%%%%%%%%%%%%%%%%%%%%%%%%%%%%%%%%%%%%%%%%%%%

\subsection{Decompositions} \ 

\medskip

Suppose $F$ is a chamber in $\ScptB$, and $E$ is a facet in $F$.  Assume $(\sigma , V_{\sigma})$ and $(\tau , V_{\tau})$ are irreducible representations of $\Gk_{E}$ with a nonzero $\Gk_{F}$ fixed vector. We fix invariant positive definite forms $\langle \, , \, \rangle_{\sigma}$ and $\langle \, , \, \rangle_{\tau}$ on $V_{\sigma}$ and $V_{\tau}$ respectively.  For any $x \, , \, y \in V_{\sigma}$ and $k \in \Gk_{E}$, define the matrix coefficient $\mymatrix{m}{\sigma}{x}{y} (k) \, := \, \langle  x \, , \, \sigma (k) y \rangle$ as in \eqref{defn-coeff}.  \  Suppose $a \, , \, b \in V_{\sigma}$.  When $v \, \in \, V_\sigma^{\Gk_{F}}$, we note that the function {\,}$\mymatrix{m}{\sigma}{a}{v}${\,} (resp.~$\mymatrix{m}{\sigma}{v}{b}$) is right (resp.~left) $\Gk_{F}$-invariant. \ We further observe:

\begin{itemize} 

\item[(i)] \ If $v \, \in \, V_\sigma^{\Gk_{F}}$ satisfies the normalization \eqref{normalization}, then {\,}${\mymatrix{m}{\sigma}{v}{v}}${\,} is both a convolution idempotent and  $\Gk_{F}$-bi-invariant.

\smallskip

\item[(ii)] \ If $\{ v_{i} \}$ is an orthogonal basis of $V_{\sigma}$ with every basis vector $v_i$ satisfying the normalization \eqref{normalization}, then the idempotents $\mymatrix{m}{\sigma}{v_i}{v_i}$
 are mutually orthogonal, and 
\begin{equation}\label{central-idempotent}
e_{\sigma} \ := \ \sum^{\deg (\sigma )}_{i=1} \mymatrix{m}{\sigma}{v_i}{v_i} \ \ \in \ \ e_{E} \, {\myprod} \, \Cic (\Gk_{E} ) \, {\myprod} \, e_{E}
\end{equation}
\noindent{is} the central idempotent attached to $\sigma$.  Set 
\begin{equation}\label{setXi}
\aligned
\Xi \ := \ \ &{\text{\rm{collection of irreducible representations}}} \ (\sigma,V_{\sigma}) \\
&{\text{\rm{of $\Gk_{E}$ which have nonzero $\Gk_{F}$-fixed vectors{\,}.}}}
\endaligned
\end{equation}
\noindent{Then}, 
\begin{equation}\label{idem-formula}
e_{E} \ = \ {\underset {\sigma \, \in \, \Xi} \sum} \ e_{\sigma} \ = \ {\underset {\sigma \, \in \, \Xi} \sum} \, \sum^{\deg (\sigma )}_{i=1} \mymatrix{m}{\sigma}{v^{\sigma}_i}{v^{\sigma}_i} \ ,
\end{equation}

\noindent{where} {\,}$\{ v^{\sigma}_{i} \}${\,} is a orthogonal basis of $V_{\sigma}$ satisfying  \eqref{normalization}.
\end{itemize}

%\medskip

\vskip 0.30in

\begin{prop}\label{p:basis}  Assume $F \subset \ScptB$ is a chamber and $E \subset F$ is a facet.  With the above notation, suppose $(\sigma ,  V_{\sigma}) \, \in \, \Xi$.  Then, there exists (finitely) $a^{\sigma}_{k} \, \in \myhecke{E}{F}$ so that 
$$
\Theta_{\sigma} \ = \ {\underset k \sum} \ \ a^{\sigma}_{k} \ {\myprod} \ (a^{\sigma}_{k})^{\staralg}  \ .
$$ 
Hence, there exists (finitely) $b_{j} \, \in \myhecke{E}{F}$ so that 
$e_{E} \ = \ {\underset j \sum} \ \ b_{j} \ {\myprod} \ b^{\staralg}_{j}$.

\end{prop}

\begin{proof}   \   We take $\{ u_i \}$ to be an orthogonal basis for $V_{\sigma}$ and $v \in V^{\Gk_{F}}_{\sigma}$ so that $\{ \mymatrix{m}{\sigma}{u_i}{u_i} \}$, and $\mymatrix{m}{\sigma}{v}{v}$ are idempotents. The coefficient  {\,}$a^{\sigma}_{k} \ = \ \mymatrix{m}{\sigma}{u_k}{v}${\,} is right $\Gk_{F}$-invariant, and 
$$
a^{\sigma}_{k} \ {\myprod} \ (a^{\sigma}_{k})^{\staralg} \ = \ \mymatrix{m}{\sigma}{u_k}{v} \ {\myprod} \ \mymatrix{m}{\sigma}{v}{u_k} \ = \ \mymatrix{m}{\sigma}{u_k}{u_k} \ ,
$$
\noindent{and} so $\Theta_{\sigma} \, = \, {\underset k {\sum}} \ \mymatrix{m}{\sigma}{u_k}{u_k} \, = \, {\underset k {\sum}} a^{\sigma}_{k} \ {\myprod} \ (a^{\sigma}_{k})^{\staralg}$.  

\end{proof}

\vskip 0.30in

\begin{prop}\label{heckedecomp} \  For each $\sigma \in \Xi$, let $\{ \, v^{\sigma}_{i} \, \}$ and $\{ \, w^{\sigma}_{i} \, \}$ be two orthogonal bases of $V_{\sigma}$.  Assume all these vectors satisfy the normalization \eqref{normalization}.  Then, ${\CalH}_{E}$ has a direct sum decomposition
\begin{equation}\label{dsum0}
{\CalH}_{E} \ = \ {\underset {\sigma, \tau \, \in \, \Xi} \bigoplus} \ \ {\underset {\text{\small{$\begin{matrix} {\,}^{i=1{\,}} \end{matrix}$}}} {\overset {\deg (\sigma )} {\bigoplus}}}  \ \ {\underset {\text{\small{$\begin{matrix} {\,}^{j=1} \end{matrix}$}}} {\overset {\deg (\tau )} {\bigoplus}}}
\ \ \mymatrix{m}{\sigma}{v^{\sigma}_i}{v^{\sigma}_i} \ {\myprod} \ \Cic (\Gk ) \ {\myprod} \ \mymatrix{m}{\tau}{w^{\tau}_j}{w^{\tau}_j} \ \ .
\end{equation}
\noindent{In} particular, any $f \in {\CalH}_{E}$ can be written uniquely as: 

\begin{equation}\label{dsum1}
f \ = {\underset {\sigma \, \in \, \Xi} \sum} \, \sum^{\deg (\sigma )}_{i=1} {\underset {\tau \, \in \, \Xi} \sum} \, \sum^{\deg (\tau )}_{j=1} f_{\sigma , i , \tau , j} \ ,
\end{equation}
\noindent{where} $f_{\sigma , i , \tau , j} \, = \, \mymatrix{m}{\sigma}{v^{\sigma}_i}{v^{\sigma}_i} \ {\myprod} \ f \ {\myprod} \ \mymatrix{m}{\tau}{w^{\tau}_j}{w^{\tau}_j}$.
\end{prop}

\begin{proof} \ Suppose $f \in {\CalH}_{E}$.  Since $f = e_{E} \myprod f \myprod e_{E}$, the decomposition \eqref{idem-formula} of $e_{E}$ then yields the sum \eqref{dsum1}, i.e., ${\CalH}_{E}$ is a sum of the indicated subspaces in \eqref{dsum0}.  To see the sum is direct, we note that convolution on the left by $\mymatrix{m}{\sigma}{v^{\sigma}_{i}}{v^{\sigma}_{i}}$ and the right by $\mymatrix{m}{\sigma}{w^{\tau}_{j}}{w^{\tau}_{j}}$ is zero on $\mymatrix{m}{\kappa}{v^{\kappa}_{r}}{v^{\kappa}_{r}} \ {\myprod} \ \Cic (\Gk ) \ {\myprod} \ \mymatrix{m}{\lambda}{w^{\lambda}_{s}}{w^{\lambda}_{s}}$ unless $(\sigma , i, \tau , j) = (\kappa , r , \lambda ,s)$, and  is the identity (since $v^{\sigma}_{i}$, $v^{\tau}_{j}$ are properly normalized) on $\mymatrix{m}{\sigma}{v^{\sigma}_i}{v^{\sigma}_i} \ {\myprod} \ \Cic (\Gk ) \ {\myprod} \ \mymatrix{m}{\tau}{w^{\tau}_j}{w^{\tau}_j}$.  Thus, the sum is direct.

\end{proof}

%%%%%%%%%%%%%%%%%%%%%%%%%%%%%%%%%%%%%%%%%%%%%%%%%%%%%%%%%%%%%%%%%%%%%%%%%%%%%%%
%%%%%%%%%%%%%%%%%%%%%%%%%%%%%%%%%%%%%%%%%%%%%%%%%%%%%%%%%%%%%%%%%%%%%%%%%%%%%%%%
%%%%%%%%%%%%%%%%%%%%%%%%%%%%%%%%%%%%%%%%%%%%%%%%%%%%%%%%%%%%%%%%%%%%%%%%%%%%%%%
%%%%%%%%%%%%%%%%%%%%%%%%%%%%%%%%%%%%%%%%%%%%%%%%%%%%%%%%%%%%%%%%%%%%%%%%%%%%%%

\vskip 0.30in

%%%%%%%%%%%%%%%%%%%%%%%%%%%%%%%%%%%%%%%%%%%%%%%%%%%%%%%%%%%%%%%%%%%%%%%%%
%%%%%%%%%%%%%%%%%%%%%%%%%%%%%%%%%%%%%%%%%%%%%%%%%%%%%%%%%%%%%%%%%%%%%%%%%
%%%%%%%%%%%%%%%%%%%%%%%%%%%%%%%%%%%%%%%%%%%%%%%%%%%%%%%%%%%%%%%%%%%%%%%%%

\section{Morita Equivalence}

\medskip

Let ${\frcC} ({\CalH}_{E})$,  ${\frcC}_{\text{\rm{fg}}} ({\CalH}_{E})$, and ${\frcC}_{\text{\rm{fin}}} ({\CalH}_{E})$ denote the categories of left, left finitely generated, and left finite-dimensional  ${\CalH}_{E}$-modules respectively, and ${\frcC} ({\CalH}_{F})$, ${\frcC}_{\text{\rm{fg}}} ({\CalH}_{F})$, and ${\frcC}_{\text{\rm{fin}}} ({\CalH}_{F})$ the analogous categories of (left) ${\CalH}_{F}$-modules.

\begin{thm}\label{main0} \quad The algebras ${\CalH}_{E}$ and ${\CalH}_{F}$ are Morita equivalence.  We have:
\smallskip  
\begin{itemize}
\item[(i)] \ \ The two maps:
$$
\begin{array}{crc}
{\myhecke{E}{F}} \otimes_{{\CalH}_{F}} {\myhecke{F}{E}} &\longrightarrow &{\CalH}_{E} \matrixskip \\
f \otimes_{{\CalH}_{F}} g &\longrightarrow  &{f \myprod g} 
\end{array}
\qquad {\text{\rm{and}}} \qquad 
\begin{array}{crc}
{\myhecke{F}{E}} \otimes_{{\CalH}_{E}} {\myhecke{E}{F}} &\longrightarrow  &{\CalH}_{F} \matrixskip \\
g \otimes_{{\CalH}_{E}} f   &\longrightarrow  &{g \myprod f} 
\end{array}
$$
\noindent are isomorphisms.

\smallskip

\item[(ii)] \ \ The maps
$$
\begin{array}{ccc}
{{\frcC} ({\CalH}_{F})} &{\longrightarrow} &{{\frcC} ({\CalH}_{E})} \matrixskip \\ 
{X} &{\longrightarrow}  &{{\myhecke{E}{F}} \otimes_{{\CalH}_{F}} X}
\end{array}
\qquad {\text{\rm{and}}} \ \ \quad 
\begin{array}{ccc}
{{\frcC} ({\CalH}_{E})} &{\longrightarrow} &{{\frcC} ({\CalH}_{F})} \matrixskip \\ 
{Y} &\longrightarrow &{{\myhecke{F}{E}} \otimes_{{\CalH}_{E}} Y}
\end{array}
$$

\smallskip

\noindent{are} inverses to each other, and give an equivalence between the categories of left ${\CalH}_{E}$-modules and left ${\CalH}_{F}$-modules.
\medskip

\item[(iii)] \ \ The category equivalences of part (i) restricts to equivalences between \ ${\frcC}_{\text{\rm{fg}}} ({\CalH}_{E})${\,} and {\,}${\frcC}_{\text{\rm{fg}}} ({\CalH}_{F})$ and between ${\frcC}_{\text{\rm{fin}}} ({\CalH}_{E})${\,} and {\,}${\frcC}_{\text{\rm{fin}}} ({\CalH}_{F})$.

\end{itemize}
\end{thm}

\begin{proof} \quad  The statements follow from the fact that {\,}$e_{F} \in {\CalH}_{E}${\,} is a full idempotent (${\CalH}_{E} \, = \, {\CalH}_{E} \myprod e_{F} \myprod {\CalH}_{E}$) \ (see Chapter 18 \cite{Lam}).

\end{proof}

We a grateful to Konstantin Ardakov whom brought our attention to the fruitfulness of establishing Morita equivalence via full idempotents.

\medskip

We remark there is a similar Morita equivalence between ${\CalH}^{\text{\rm{fin}}}_{F}$ and ${\CalH}^{\text{\rm{fin}}}_{E}$.

\bigskip

The Morita equivalence of ${\CalH}_{E}$ and ${\CalH}_{F}$ means their
centers are isomorphic.   The center of the Peter--Weyl algebra
${\CalH}_{E}$ can be obtained from the (well known) center of the Iwahori Hecke algebra ${\CalH}_{F}$ via the following result.

\begin{cor}\label{main1} Express $e_{E}$ as $e_{E} \, = \, {\underset {i=1} {\overset{r} \sum}} \ a_{i} \myprod b_i$ with $a_{i} \in {\myhecke{E}{F}}$ and $b_i \ \in {\myhecke{F}{E}}$.  Then, the isomorphism of centers in the Morita equivalence of ${\CalH}_{F}$ and ${\CalH}_{E}$ is given by 
$$
z \ \ \longrightarrow \ \ \sum^{r}_{i=1} \ a_{i} {\,}\myprod{\,} z  {\,}\myprod{\,} b_i \ \ .
$$
\end{cor}

\medskip

\begin{proof} \ \ If $z$ is in the center of ${\CalH}_{F}$, then for any $X \in {\frcC} ({\CalH}_{F})$ the map $x \, \rightarrow \, zx$ commutes with any self-morphism of $X$, i.e., it is in the center of the category.  Under the functor 
$X \rightarrow {\myhecke{E}{F}} \otimes_{{\CalH}_{F}} X$, we have a similar self-morphism  $(f \otimes_{{\CalH}_{F}} x) \rightarrow (f \otimes_{{\CalH}_{F}} zx)$ of ${\myhecke{E}{F}} \otimes_{{\CalH}_{F}} X$ in the category ${\frcC} ({\CalH}_{E})$. We commute
$$
\aligned
f \otimes_{{\CalH}_{F}} zx \ &= \ (e_{E} \myprod f)  \otimes_{{\CalH}_{F}} zx \ = \ {\underset {i=1} {\overset{r} \sum}} \ a_{i} \myprod b_i \myprod f \otimes_{{\CalH}_{F}} zx \\ 
&= \ {\underset {i=1} {\overset{r} \sum}} \ a_{i} \otimes_{{\CalH}_{F}}  (b_i \myprod f)\, zx \qquad {\text{\rm{(since $b_i \myprod f \, \in \, {\CalH}_{F}$)}}} \\ 
&= \ {\underset {i=1} {\overset{r} \sum}} \ a_{i} \otimes_{{\CalH}_{F}}  z \myprod (b_i \myprod f) \, x \ = \  {\underset {i=1} {\overset{r} \sum}} \ a_{i} \myprod (z \myprod (b_i \myprod f)) \otimes_{{\CalH}_{F}}     x \\
&= \ \big( \, {\underset {i=1} {\overset{r} \sum}} \ a_{i} \myprod z \myprod b_i \, \big) (f \otimes_{{\CalH}_{F}} x) \ \ .
\endaligned
$$ 
\noindent So, the element ${\underset {i=1} {\overset{r} \sum}} \ a_{i} \myprod z \myprod b_i${\,} is the central element of ${\CalH}_{E}$ corresponding to $z$.
\end{proof}
\vskip 0.300in

%%%%%%%%%%%%%%%%%%%%%%%%%%%%%%%%%%%%%%%%%%%%%%%%%%%%%%%%%%%%%%%%%%%%%%%%%
%%%%%%%%%%%%%%%%%%%%%%%%%%%%%%%%%%%%%%%%%%%%%%%%%%%%%%%%%%%%%%%%%%%%%%%%%
%%%%%%%%%%%%%%%%%%%%%%%%%%%%%%%%%%%%%%%%%%%%%%%%%%%%%%%%%%%%%%%%%%%%%%%%%

\section{Involutions and Forms}

\smallskip

\subsection{Extension of an anti-involution of ${\CalH}_{F}$ to ${\CalH}_{E}$} \ \ 

\medskip

We continue with the assumption $F \subset \ScptB$ is a chamber and $E
\subset F$ a facet.  Let $\staralg$ be the anti-involution
\eqref{cstar}.   Suppose the Iwahori-Hecke algebra ${\CalH}_{F}$ has
an anti-involution $\circ$ satisfying 
\begin{equation}\label{circ}
\forall \ f \ \in e_{F} \myprod \Cic (\Gk_{E} ) \myprod e_{F} \quad : \quad f^{\circ} \ = \ f^{\staralg}
\end{equation}
\noindent{We} show here, that it is possible to extend the anti-involution $\circ$ of ${\CalH}_{F}$ to an anti-involution of ${\CalH}_{E}$.

\medskip

\begin{lemma}\label{submodule0} \ For each $\kappa \, \in \, \Xi$,
  choose two bases $\{ v^{\kappa}_{i} \}$, and $\{ w^{\kappa}_{i} \}$ of $V_{\kappa}$, and choose two elements $y^{\kappa}, z^{\kappa} \, \in \, V^{\Gk_{F}}_{\kappa}$  satisfying the normalization \eqref{normalization}.     \ Then, 
\begin{itemize}
\item[(i)] \ The $\Gk^{F}$-bi-invariant function $\mymatrix{m}{\sigma}{y^{\sigma}}{v^{\sigma}_{i}} \, \myprod \, f \, \myprod \mymatrix{m}{\tau}{w^{\tau}_{j}} {z^{\tau}}$ is convolution left invariant for $\mymatrix{m}{\sigma}{y^{\sigma}}{y^{\sigma}}$ and convolution right invariant for $\mymatrix{m}{\tau}{z^{\tau}}{z^{\tau}}$.

\smallskip

\item[(ii)] \ $\forall \ f \ \in \Cic (\Gk )$, and $\sigma, \, \tau \, \in \, \Xi$: 
$$
{\mymatrix{m}{\sigma}{v^{\sigma}_{i}}{v^{\sigma}_{i}}} \myprod f \myprod {\mymatrix{m}{\tau}{w^{\tau}_{j}}{w^{\tau}_{j}}} \ = \ {\mymatrix{m}{\sigma}{v^{\sigma}_{i}}{y^{\sigma}}} \myprod F_{y^{\sigma} , z^{\tau}} \myprod {\mymatrix{m}{\tau}{z^{\tau}}{w^{\tau}_{i}}} \ ,
$$
\noindent{where} $F_{y^{\sigma} , z^{\tau}} \, = \, \mymatrix{m}{\sigma}{y^{\sigma}}{v^{\sigma}_{i}} \, \myprod \, f \, \myprod \mymatrix{m}{\tau}{w^{\tau}_{j}} {z^{\tau}}$ belongs to ${\CalH}_{F}$.
\end{itemize}
\end{lemma}
\smallskip
\begin{proof} \quad Clear.
\end{proof}

\medskip

\noindent{Remarks.} \ 
\begin{itemize} 
\item[$\bullet$] \ A consequence of statement (ii) is that 
$$
{\CalH}_{E} \ = \ ( \ e_{E} \myprod \Cic (\Gk_{E}) \myprod e_{F} \ ) \ \myprod \ {\CalH}_{F} \ \myprod \ ( \ e_{F} \myprod \Cic (\Gk_{E}) \myprod e_{E} \ ) \ .
$$
\item[$\bullet$] \ In statement (ii), if we replace the collection of (normalized) $\Gk_{F}$-invariant vectors $\{ y^{\kappa} \}$ and $\{ z^{\kappa} \}$ by $\{ y^{\kappa}_{\dagger} \}$ and $\{ z^{\kappa}_{\dagger} \}$, then the two $\Gk_{F}$-bi-invariant functions $F_{y^{\sigma} , z^{\tau}}$ and $F_{y^{\sigma}_{\dagger} , z^{\tau}_{\dagger}}$ are related by
\begin{equation}
F_{y^{\sigma}_{\dagger} , z^{\tau}_{\dagger}} \ = \ \mymatrix{m}{\sigma}{y^{\sigma}_{\dagger}}{y^{\sigma}} \ {\myprod} \ F_{y^{\sigma} , z^{\tau}} \ {\myprod} \ \mymatrix{m}{\sigma}{z^{\tau}}{z^{\tau}_{\dagger}} \ .
\end{equation}
\end{itemize}

\smallskip

Assume we are in the situation of Lemma \ref{submodule0}.  Then, any $f \in {\CalH}_{E}$, is decomposed as in \eqref{dsum1}; thus, 
$$
f \ = {\underset {\sigma \, \in \, \Xi} \sum} \, \sum^{\deg (\sigma )}_{i=1} {\underset {\tau \, \in \, \Xi} \sum} \, \sum^{\deg (\tau )}_{j=1} 
\ \mymatrix{m}{\sigma}{v^{\sigma}_i}{y^{\sigma}} \ {\myprod} \ \big( \ \mymatrix{m}{\sigma}{y^{\sigma}}{v^{\sigma}_i} \ {\myprod} \ f \ {\myprod} \ \mymatrix{m}{\tau}{v^{\tau}_j}{z^{\tau}} \ \big) \  {\myprod} \ \mymatrix{m}{\tau}{z^{\tau}}{v^{\tau}_j} \
$$
\noindent{and} each function {\,}$\big( \, \mymatrix{m}{\sigma}{y^{\sigma}}{v^{\sigma}_i} \ {\myprod} \ f \ {\myprod} \ \mymatrix{m}{\tau}{v^{\tau}_j}{z^{\tau}} \, \big)${\,} is $\Gk_{F}$-bi-invariant.  Another choice $\{ \, y^{\sigma}_{\dagger} , \, z^{\sigma}_{\dagger}  \, \in \, V^{\Gk_{F}}_{\sigma} \ | \ \sigma \in \Xi \, \}$ yields 
$$
f \ = {\underset {\sigma \, \in \, \Xi} \sum} \, \sum^{\deg (\sigma )}_{i=1} {\underset {\tau \, \in \, \Xi} \sum} \, \sum^{\deg (\tau )}_{j=1} 
\ \mymatrix{m}{\sigma}{v^{\sigma}_i}{y^{\sigma}_{\dagger}} \ {\myprod} \ \big( \ \mymatrix{m}{\sigma}{y^{\sigma}_{\dagger}}{v^{\sigma}_i} \ {\myprod} \ f \ {\myprod} \ \mymatrix{m}{\tau}{v^{\tau}_j}{z^{\tau}_{\dagger}} \ \big) \  {\myprod} \ \mymatrix{m}{\tau}{z^{\tau}_{\dagger}}{v^{\tau}_j} \
$$

\noindent{We} can combine these two expressions for $f$ with:
\begin{itemize}
\item[(i)] \ The assumption $\circ$ is an anti-involution of the Iwahori-Hecke algebra ${\CalH}_{F}$.
\smallskip
\item[(ii)] \ On the functions $e_{F} \, {\myprod} \, C(\Gk_{E}) \, {\myprod} \, e_{F}$, the maps $\circ$ and $\staralg$ (of \eqref{cstar}) are equal.
\smallskip
\item[(iii)] \ For any $\kappa \in \Xi$ and $a, b \in V_{\kappa}$, \ $(\mymatrix{m}{\kappa}{a}{b})^{\staralg} \ = \ \mymatrix{m}{\kappa}{b}{a}$.

\begin{comment}
\item[(iv)] $\big( \ \mymatrix{m}{\sigma}{u^{\sigma}}{v^{\sigma}_i} \ {\myprod} \ f \ {\myprod} \ \mymatrix{m}{\tau}{v^{\tau}_j}{u^{\tau}} \ \big) \ = \ \mymatrix{m}{\sigma}{u^{\sigma}}{u^{\sigma}_{\dagger}} \ {\myprod} \ \big( \ \mymatrix{m}{\sigma}{u^{\sigma}_{\dagger}}{v^{\sigma}_i} \ {\myprod} \ f \ {\myprod} \ \mymatrix{m}{\tau}{v^{\tau}_j}{u^{\tau}_{\dagger}} \ \big) \ {\myprod} \ \mymatrix{m}{\tau}{u^{\tau}_{\dagger}}{u^{\tau}}$, 
and 
$\mymatrix{m}{\sigma}{u^{\sigma}_{\dagger}}{v^{\sigma}_i} , \, \mymatrix{m}{\tau}{u^{\tau}_{\dagger}}{u^{\tau}} \, \in \, e_{F} \, {\myprod} \, C(\Gk_{E}) \, {\myprod} \, e_{F}$ \ .
\end{comment}

\end{itemize}

\smallskip

\noindent{We} deduce that  the linear map
\begin{equation}\label{heckeinvolution}
\aligned
f \ &= {\underset {\sigma \, \in \, \Xi} \sum} \, \sum^{\deg (\sigma )}_{i=1} {\underset {\tau \, \in \, \Xi} \sum} \, \sum^{\deg (\tau )}_{j=1} 
\ \mymatrix{m}{\sigma}{v^{\sigma}_i}{u^{\sigma}} \ {\myprod} \ \big( \ \mymatrix{m}{\sigma}{u^{\sigma}}{v^{\sigma}_i} \ {\myprod} \ f \ {\myprod} \ \mymatrix{m}{\tau}{v^{\tau}_j}{u^{\tau}} \ \big) \  {\myprod} \ \mymatrix{m}{\tau}{u^{\tau}}{v^{\tau}_j} \\
& \qquad {\overset {{\ }_{\circ}{\ }} {\xrightarrow{\hskip 0.30in}} } \quad f^{\bullet} \ := {\underset {\sigma \, \in \, \Xi} \sum} \, \sum^{\deg (\sigma )}_{i=1} {\underset {\tau \, \in \, \Xi} \sum} \, \sum^{\deg (\tau )}_{j=1} 
\ 
\mymatrix{m}{\tau}{v^{\tau}_j}{u^{\tau}} \ {\myprod} \ \big( \ \mymatrix{m}{\sigma}{u^{\sigma}}{v^{\sigma}_i} \ {\myprod} \ f \ {\myprod} \ \mymatrix{m}{\tau}{v^{\tau}_j}{u^{\tau}} \ \big)^{\circ} \  {\myprod} \ \mymatrix{m}{\sigma}{u^{\sigma}}{v^{\tau}_i} \\
&{\hskip 1.135in} = {\underset {\sigma \, \in \, \Xi} \sum} \, \sum^{\deg (\sigma )}_{i=1} {\underset {\tau \, \in \, \Xi} \sum} \, \sum^{\deg (\tau )}_{j=1} 
\ 
\mymatrix{m}{\tau}{v^{\tau}_j}{u^{\tau}} \ {\myprod} \ \big( \ \mymatrix{m}{\tau}{u^{\tau}}{v^{\tau}_j} \ {\myprod} \ f^{\circ} \ {\myprod} \ \mymatrix{m}{\sigma}{v^{\sigma}_i}{u^{\sigma}} \ \big) \  {\myprod} \ \mymatrix{m}{\sigma}{u^{\sigma}}{v^{\tau}_i} \\
\endaligned
\end{equation}
\noindent{on} ${\CalH}_{E}$ is well-defined.  

%\eject

\begin{prop}\label{involution-prop} \ The linear map {\,}$\circ${\,} \eqref{heckeinvolution} of ${\CalH}_{E}$ is an algebra anti-involution.
\end{prop}

\noindent{We} note:
\begin{itemize} 
\item[(i)] \ If $\circ$ is the $\staralg$ anti-involution of ${\CalH}_{F}$, then the extension $\circ$ to  ${\CalH}_{F}$ is the  $\staralg$ anti-involution.
\item[(ii)] \ Since the {\,}$\bullet${\,} anti-involution satisfies \eqref{circ}, it has an extension to ${\CalH}_{E}$.
\end{itemize}

\begin{proof}  \   For each $\sigma \in \Xi$, we fix an orthogonal basis $\{ v^{\sigma}_{i}$ of $V_{\sigma}$ and a vector $u^{\sigma} \in V^{\Gk_{F}}_{sigma}$.  We assume the vectors are normalized as in \eqref{normalization}.  Suppose $f , \, g \, \in {\CalH}_{E}$.  Expand them  as 
$$
f \ = \ {\underset {\sigma \, , \, \tau} \sum} \ {\sum^{\deg (\sigma )}_{i=1}} \ {\sum^{\deg (\tau )}_{j=1}}  \mymatrix{m}{\sigma}{v^{\sigma}_{i}}{v^{\sigma}_{i}} \, {\myprod} \, f \, {\myprod} \, \mymatrix{m}{\tau}{v^{\tau}_{j}}{v^{\tau}_{j}}
\quad {\text{\rm{and}}} \quad g \ = \ {\underset {\kappa \, , \, \lambda} \sum} \ {\sum^{\deg (\kappa )}_{r=1}} \ {\sum^{\deg (\lambda )}_{s=1}}  \mymatrix{m}{\kappa}{v^{\kappa}_{r}}{v^{\kappa}_{r}} \, {\myprod} \, g \, {\myprod} \, \mymatrix{m}{\lambda}{v^{\lambda}_{s}}{v^{\lambda}_{s}} \ .
$$

\eject 

\noindent{By} the orthogonality relations:

$$
\aligned
\big( \mymatrix{m}{\sigma}{v^{\sigma}_{i}}{v^{\sigma}_{i}} \, {\myprod} \, f \, {\myprod} \, \mymatrix{m}{\tau}{v^{\tau}_{j}}{v^{\tau}_{j}} \big) \ &\myprod \ \big( \mymatrix{m}{\kappa}{v^{\kappa}_{r}}{v^{\kappa}_{r}} \, {\myprod} \, g \, {\myprod} \, \mymatrix{m}{\lambda}{v^{\lambda}_{s}}{v^{\lambda}_{s}} \big) \\
&\ \\
&= \ 
\begin{cases}
\ 0 &{\text{\rm{unless {\,}$\tau = \kappa${\,} and $j=r$}}} \\
\ &\  \\
\ {\underset {\tau} \sum} \ {{\overset {\deg (\tau )} {\underset {j=1} \sum}}} \ \mymatrix{m}{\sigma}{v^{\sigma}_{i}}{v^{\sigma}_{i}} \, {\myprod} \, f \, {\myprod} \, \mymatrix{m}{\tau}{v^{\tau}_{j}}{v^{\tau}_{j}} {\myprod} \, g \, {\myprod} \, \mymatrix{m}{\lambda}{v^{\lambda}_{s}}{v^{\lambda}_{s}}
&{\text{\rm{when {\,}$\tau = \kappa${\,} and $j=r$}}} \\
\end{cases}
\endaligned
$$

\medskip

\noindent{By} \eqref{dsum1}, this must be $\mymatrix{m}{\sigma}{v^{\sigma}_{i}}{v^{\sigma}_{i}} \, {\myprod} \ f \ {\myprod} \, g \ {\myprod} \, \, \mymatrix{m}{\lambda}{v^{\lambda}_{s}}{v^{\lambda}_{s}}$.  From this, we use $\mymatrix{m}{\sigma}{v^{\sigma}_{i}}{v^{\sigma}_{i}} =  \mymatrix{m}{\sigma}{v^{\sigma}_{i}}{u^{\sigma}} \mymatrix{m}{\sigma}{u^{\sigma}}{v^{\sigma}_{i}}$, $\mymatrix{m}{\tau}{v^{\tau}_{j}}{v^{\tau}_{j}} =  \mymatrix{m}{\tau}{v^{\tau}_{j}}{u^{\tau}} \mymatrix{m}{\tau}{u^{\tau}}{v^{\tau}_{j}}$, and $\mymatrix{m}{\lambda}{v^{\lambda}_{s}}{v^{\lambda}_{s}} =  \mymatrix{m}{\lambda}{v^{\lambda}_{s}}{u^{\lambda}} \mymatrix{m}{\lambda}{u^{\lambda}}{v^{\lambda}_{s}}$ to compute

\smallskip

$$
\aligned
\big( \ &\mymatrix{m}{\sigma}{v^{\sigma}_{i}}{v^{\sigma}_{i}} \ {\myprod} \ f \ {\myprod} \, g \ {\myprod} \ \mymatrix{m}{\lambda}{v^{\lambda}_{s}}{v^{\lambda}_{s}} \ \ \big)^{\circ} \ = \  \Big( {\underset {\tau} \sum} \ {\sum^{\deg (\tau )}_{j=1}}
\mymatrix{m}{\sigma}{v^{\sigma}_{i}}{v^{\sigma}_{i}} \, {\myprod} \, f \, {\myprod} \, \mymatrix{m}{\tau}{v^{\tau}_{j}}{v^{\tau}_{j}} {\myprod} \, g \, {\myprod} \, \mymatrix{m}{\lambda}{v^{\lambda}_{s}}{v^{\lambda}_{s}} \Big)^{\circ} \\
&= \  \Big( {\underset {\tau} \sum} \ {\sum^{\deg (\tau )}_{j=1}}
\mymatrix{m}{\sigma}{v^{\sigma}_{i}}{u^{\sigma}} \, \myprod \, \mymatrix{m}{\sigma}{u^{\sigma}}{v^{\sigma}_{i}} \, {\myprod} \, f \, {\myprod} \, \mymatrix{m}{\tau}{v^{\tau}_{j}}{v^{\tau}_{j}} {\myprod} \, g \, {\myprod} \, \mymatrix{m}{\lambda}{v^{\lambda}_{s}}{u^{\lambda}} \, \myprod \, \mymatrix{m}{\lambda}{u^{\lambda}}{v^{\lambda}_{s}} \Big)^{\circ} \\
&= \  {\underset {\tau} \sum} \ {\sum^{\deg (\tau )}_{j=1}} \ 
\mymatrix{m}{\lambda}{v^{\lambda}_{s}}{u^{\lambda}} \, \myprod \ \big( \mymatrix{m}{\sigma}{u^{\sigma}}{v^{\sigma}_{i}} \, {\myprod} \, f \, {\myprod} \, \mymatrix{m}{\tau}{v^{\tau}_{j}}{v^{\tau}_{j}} {\myprod} \, g \, {\myprod} \mymatrix{m}{\lambda}{v^{\lambda}_{s}}{u^{\lambda}} \big)^{\circ} \, \myprod \,  \mymatrix{m}{\sigma}{u^{\sigma}}{v^{\sigma}_{i}} \\
&= \  {\underset {\tau} \sum} \ {\sum^{\deg (\tau )}_{j=1}} \
\mymatrix{m}{\lambda}{v^{\lambda}_{s}}{u^{\lambda}} \, \myprod \ \big( 
\mymatrix{m}{\sigma}{u^{\sigma}}{v^{\sigma}_{i}} \, {\myprod} \, f \, {\myprod} \,
\mymatrix{m}{\tau}{v^{\tau}_{j}}{u^{\tau}} \, \myprod \, 
\mymatrix{m}{\tau}{u^{\tau}}{v^{\tau}_{j}} \, {\myprod} \, g \, {\myprod} \, 
\mymatrix{m}{\lambda}{v^{\lambda}_{s}}{u^{\lambda}} \big)^{\circ}
 \, \myprod \,  \mymatrix{m}{\sigma}{u^{\sigma}}{v^{\sigma}_{i}} \\
&= \  {\underset {\tau} \sum} \ {\sum^{\deg (\tau )}_{j=1}} \
\mymatrix{m}{\lambda}{v^{\lambda}_{s}}{u^{\lambda}} \, \myprod \ \big( 
(\mymatrix{m}{\tau}{u^{\tau}}{v^{\tau}_{j}} \, {\myprod} \, g \, {\myprod} \, 
\mymatrix{m}{\lambda}{v^{\lambda}_{s}}{u^{\lambda}})^{\circ} 
\, \myprod \, 
(\mymatrix{m}{\sigma}{u^{\sigma}}{v^{\sigma}_{i}} \, {\myprod} \, f \, {\myprod} \,
\mymatrix{m}{\tau}{v^{\tau}_{j}}{u^{\tau}} )^{\circ} \big)
 \, \myprod \,  \mymatrix{m}{\sigma}{u^{\sigma}}{v^{\sigma}_{i}} \\
%%%%%%%% A
&= \  {\underset {\tau} \sum} \ {\sum^{\deg (\tau )}_{j=1}} \
\mymatrix{m}{\lambda}{v^{\lambda}_{s}}{u^{\lambda}} \, \myprod \ \big( 
(\mymatrix{m}{\tau}{u^{\tau}}{v^{\tau}_{j}} \, {\myprod} \, g \, {\myprod} \, 
\mymatrix{m}{\lambda}{v^{\lambda}_{s}}{u^{\lambda}})^{\circ} 
\, \myprod \, \mymatrix{m}{\tau}{u^{\tau}}{u^{\tau}} \, \myprod \, 
(\mymatrix{m}{\sigma}{u^{\sigma}}{v^{\sigma}_{i}} \, {\myprod} \, f \, {\myprod} \,
\mymatrix{m}{\tau}{v^{\tau}_{j}}{u^{\tau}} )^{\circ} \big)
 \, \myprod \,  \mymatrix{m}{\sigma}{u^{\sigma}}{v^{\sigma}_{i}} \\
 %%%%%%%% B
&= \  {\underset {\tau} \sum} \ {\sum^{\deg (\tau )}_{j=1}} \
\mymatrix{m}{\lambda}{v^{\lambda}_{s}}{u^{\lambda}} \, \myprod \ \big( 
(\mymatrix{m}{\tau}{u^{\tau}}{v^{\tau}_{j}} \, {\myprod} \, g \, {\myprod} \, 
\mymatrix{m}{\lambda}{v^{\lambda}_{s}}{u^{\lambda}})^{\circ} 
\, \myprod \, \mymatrix{m}{\tau}{u^{\tau}}{v^{\tau}_{j}} \\
&{\hskip 2.50in} \myprod \, \mymatrix{m}{\tau}{v^{\tau}_{j}}{u^{\tau}} 
\, \myprod \, 
(\mymatrix{m}{\sigma}{u^{\sigma}}{v^{\sigma}_{i}} \, {\myprod} \, f \, {\myprod} \,
\mymatrix{m}{\tau}{v^{\tau}_{j}}{u^{\tau}} )^{\circ} \big)
 \, \myprod \,  \mymatrix{m}{\sigma}{u^{\sigma}}{v^{\sigma}_{i}} \\
%%%%%%%% C
&= \  {\underset {\tau} \sum} \ {\sum^{\deg (\tau )}_{j=1}} \ {\sum^{\deg (\tau )}_{r=1}} \ 
\mymatrix{m}{\lambda}{v^{\lambda}_{s}}{u^{\lambda}} \, \myprod \ \big( 
(\mymatrix{m}{\tau}{u^{\tau}}{v^{\tau}_{r}} \, {\myprod} \, g \, {\myprod} \, 
\mymatrix{m}{\lambda}{v^{\lambda}_{s}}{u^{\lambda}})^{\circ} 
\, \myprod \, \mymatrix{m}{\tau}{u^{\tau}}{v^{\tau}_{r}} \\
&{\hskip 2.50in} \myprod \, \mymatrix{m}{\tau}{v^{\tau}_{j}}{u^{\tau}} 
\, \myprod \, 
(\mymatrix{m}{\sigma}{u^{\sigma}}{v^{\sigma}_{i}} \, {\myprod} \, f \, {\myprod} \,
\mymatrix{m}{\tau}{v^{\tau}_{j}}{u^{\tau}} )^{\circ} \big)
\, \myprod \,  \mymatrix{m}{\sigma}{u^{\sigma}}{v^{\sigma}_{i}} \\
%%%%%%%% D
{\hskip 0.20in}&= \ \big( {\underset {\kappa} \sum} \ {\sum^{\deg (\kappa )}_{r=1}} \  \mymatrix{m}{\lambda}{v^{\lambda}_{s}}{u^{\lambda}} \, \myprod \, 
(\mymatrix{m}{\kappa}{u^{\kappa}}{v^{\kappa}_{r}} \, {\myprod} \, g \, {\myprod} \, 
\mymatrix{m}{\lambda}{v^{\lambda}_{s}}{u^{\lambda}})^{\circ} 
\, \myprod \, \mymatrix{m}{\kappa}{u^{\kappa}}{v^{\kappa}_{r}} \, \big) \, {\hskip 2.90in} \\ 
&{\hskip 2.50in} \myprod \, \big( \, {\underset {\tau} \sum} \ {\sum^{\deg (\tau )}_{j=1}} \ \mymatrix{m}{\tau}{v^{\tau}_{j}}{u^{\tau}} 
\, \myprod \, 
(\mymatrix{m}{\sigma}{u^{\sigma}}{v^{\sigma}_{i}} \, {\myprod} \, f \, {\myprod} \,
\mymatrix{m}{\tau}{v^{\tau}_{j}}{u^{\tau}} )^{\circ} \, \myprod \,  \mymatrix{m}{\sigma}{u^{\sigma}}{v^{\sigma}_{i}} \big) \\
\endaligned
$$

%\eject

$$
\aligned 
%{\hskip 0.50in}
%%%%%%%% E
&= \ \big( \ g \, \myprod \, \mymatrix{m}{\lambda}{v^{\lambda}_{s}}{v^{\lambda}_{s}} \ \big)^{\circ} \ \myprod \ \big( \ \mymatrix{m}{\sigma}{v^{\sigma}_{i}}{v^{\sigma}_{i}} \, {\myprod} \, f \ \big)^{\circ}  \ {\hskip 3.50in} .  \\
\endaligned
$$

%\medskip

\noindent The above is true for any $\mymatrix{m}{\sigma}{v^{\sigma}_{i}}{v^{\sigma}_{i}}$ and $\mymatrix{m}{\lambda}{v^{\lambda}_{s}}{v^{\lambda}_{s}}$.   We conclude $\circ$ is an algebra anti-involution. 
\end{proof}

\bigskip

We note that the $\circ$-involution of ${\CalH}_{E}$ interchanges the two 
subspaces ${\myhecke{E}{F}}$ and ${\myhecke{F}{E}}$.  

\noindent{We} obviously have
$$
\forall \ \ f \, \in \, {\CalH}_{E} \ \ , \ \ a \, \in \, {\myhecke{E}{F}} \ \ , \ \ g \, \in \, {\CalH}_{F} \quad : \quad ( \, f \, \myprod \, a \, \myprod \, g \, )^{\circ} \, = \, g^{\circ} \, \myprod \, a^{\circ} \, \myprod \, f^{\circ} 
$$
\noindent{and} a similar relation when $b \, \in \, {\myhecke{F}{E}}$ instead.  We have:
$$
\forall \ \ a \, \in \, {\myhecke{E}{F}} \ \ , \ \ b \, \in \, {\myhecke{F}{E}} \quad : \quad (a \, \myprod \, b)^{\circ} \ = \ b^{\circ} \, \myprod \, a^{\circ} \quad {\text{\rm{and}}} \quad  (b \, \myprod \, a)^{\circ} \ = \ a^{\circ} \, \myprod \, b^{\circ} \ .
$$

\medskip

%%%%%%%%%%%%%%%%%%%%%%%%%%%%%%%%%%%%%%%%%%%%%%%%%%%%%%%%%%%%%%%%%%%%%%%%%%%%%%%
%%%%%%%%%%%%%%%%%%%%%%%%%%%%%%%%%%%%%%%%%%%%%%%%%%%%%%%%%%%%%%%%%%%%%%%%%%%%%%%%
%%%%%%%%%%%%%%%%%%%%%%%%%%%%%%%%%%%%%%%%%%%%%%%%%%%%%%%%%%%%%%%%%%%%%%%%%%%%%%%
%%%%%%%%%%%%%%%%%%%%%%%%%%%%%%%%%%%%%%%%%%%%%%%%%%%%%%%%%%%%%%%%%%%%%%%%%%%%%%

We prove that the Morita equivalence of Theorem \ref{main0} preserves the $\circ$ hermitian and unitary modules.   We continue in the context that $F$ is a chamber, and a facet $E \subset F$.  

\smallskip

We follow  \cite{RF}.  Suppose ${\mathcal A}$ is a $\hvC^\staralg$-algebra.    We use $\circ$ to denote the involution of a ${\mathcal A}$.  If $a \in {\mathcal A}$, we write $a \ge 0$, if there exists  $x_1, \, \dots \, , x_n \, \in \, {\mathcal A}$ so that $a \, = \, \sum^{n}_{i=1} x^{\circ}_{i} \, x_{i}$.  

\smallskip

In the Morita equivalence of Theorem \ref{main0}, we assume $\circ$ is an anti-involution of ${\CalH}_{F}$ which satsifies \eqref{circ}, so there is an extension of $\circ$ to ${\CalH}_{E}$.  We want to be able to transfer the Hermitian structure of a representation of ${\CalH}_{F}$ to a representation of ${\CalH}_{E}$ and vice-versa.  To effect this, $\myhecke{E}{F}$ must have a ${\CalH}_{F}$-valued form {\,}$( \, , \, )_{{\CalH}_{F}} \, : \, \myhecke{E}{F} \times \myhecke{E}{F} \longrightarrow {\CalH}_{F}$ which is sesquilinear, i.e., so that: 
$$
\aligned
\forall \quad a \, , \, b \ \in \ \myhecke{E}{F} \quad &: \quad ( \, a \, , \, b \, )_{{\CalH}_{F}} \, = \, ( \, ( \, b \, , \, a \, )_{{\CalH}_F} \, )^{\circ}  \\
\forall \ \ r \, \in \, {\CalH}_{E} \ \ , \ \ a \, , \, b \, \in \myhecke{E}{F} \quad &: \quad  ( \, r \, \myprod \, a \, , \, b \, )_{{\CalH}_{F}} \, = \, ( \, a \, , \, r^{\circ} \, \myprod \, b \, )_{{\CalH}_{F}} \ \ . 
\endaligned
$$
\smallskip
\noindent{Granted} the existence of the form $( \, , \, )_{{\CalH}_{F}}$, if $(\pi , V_{\pi}) \in \frcC ({\CalH}_{F})$ has a hermitian form $\langle \, , \, \rangle_{{\CalH}_F}$, then the ${\CalH}_{E}$-module ${\myhecke{E}{F}} \otimes_{{\CalH}_{F}} V_{\pi}$ is hermitian for the form 
$$
\langle \, f \otimes v \, , \, g \otimes w \, \rangle_{{\CalH}_{E}} \ = \ \langle \ \pi ( \, ( f , g )_{{\CalH}_{F}}  ) \ v \, , \, w \, \rangle_{{\CalH}_{F}} \ . 
$$

\noindent{This} plugs into the machinery of \cite{RF}, and it is formal that {\,}$\langle \, , \, \rangle_{{\CalH}_{E}}${\,} is a hermitian form with the appropriate invariance properties.   To go in the other direction, $\myhecke{F}{E}${\,} must have a {\,}${\CalH}_{E}$-valued sesquilinear form 
{\,}$( \, , \, )_{{\CalH}_{E}} \, : \, \myhecke{F}{E} \times \myhecke{F}{E} \longrightarrow {\CalH}_{E}$.{\,}   For our situation, the two forms are 
\begin{equation}\label{forms}
\aligned
\forall \quad a \, , \, b \ \in \ \myhecke{E}{F} \quad &: \quad ( \, a \, , \, b \, )_{{\CalH}_{F}} \, := \, {a^{\circ}} \, \myprod \, b \\
\forall \quad c \, , \, d \ \in \ \myhecke{F}{E} \quad &: \quad ( \, c \, , \, d \, )_{{\CalH}_{E}} \, := \, {c^{\circ}} \, \myprod \, d \\
\endaligned
\end{equation}

%%%%%  \noindent{The} (two) maps $\circ$ are the maps of \eqref{cstarEF} and \eqref{bulletEF}  respectively.

\medskip

To show a unitary module ($V \in \frcC ({{\CalH}_{F}})$) is taken to a unitary module ($({\myhecke{E}{F}} \otimes_{{\CalH}_{F}} V) \, \in \, \frcC ({\CalH}_{E})$), it suffices to show $(a,a)_{{\CalH}_{F}} \ge 0$ for any $a \in {\myhecke{E}{F}}$.  Similarly, if $W \in \frcC ({{\CalH}_{E}})$ is unitary, then a sufficient condition for $({\myhecke{F}{E}} \otimes_{{\CalH}_{E}} W)$ to be unitary is that  $(a,a)_{{\CalH}_{E}} \ge 0$ for all $a \in {\myhecke{F}{E}}$.   We write an element in ${\myhecke{F}{E}}$ (resp.~${\myhecke{E}{F}})$ as $a \, = \, e_{E} \, \myprod \, A \, \myprod \, e_{F}$ (resp.~ $a \, = \, e_{F} \, \myprod \, A \, \myprod \, e_{E}$) with $A \, \in \, {\CalH}_{E}$.  Then,
$$
\aligned
( \, a \, , \, a \, )_{{\CalH}_{F}} \ &= \ ( \, e_{E} \, \myprod \, A \, \myprod \, e_{F} \, , \, e_{E} \, \myprod \, A \, \myprod \, e_{F} \, )_{{\CalH}_{F}} \\
&= \ (  \, e_{E} \, \myprod \, A \, \myprod \, e_{F} \, )^{\circ} \, \myprod \, ( \, e_{E} \, \myprod \, A \, \myprod \, e_{F} \, ) \ = \  \, e_{F} \, \myprod \, A^{\circ} \, \myprod \, e_{E} \, \myprod \, A \, \myprod \, e_{F} \ \ .
\endaligned
$$

\noindent{By} Proposition \ref{p:basis}, there exists $x_1, \, \dots \, , x_r \, \in \, e_{E} \myprod \Cic (\Gk_{E}) \myprod e_{F}$ so that $e_{E} \, = \, \sum^{r}_{i=1} \, x_{i} \, \myprod \, x^{\staralg}_{i}$.  Substitution yields
$$
\aligned
( \, a \, , \, a \, )_{{\CalH}_{F}} \ &= \ e_{F} \, \myprod \, A^{\circ} \, \myprod \, ( \, \sum^{r}_{i=1} x_{i} \, \myprod \, x^{\staralg}_{i} \, ) \, \myprod \, A \, \myprod \, e_{F} \\ 
&= \ \sum^{r}_{i=1} \ ( \, x^{\staralg}_{i} \, \myprod \, A \, \myprod \, e_{F} \, )^{\circ} \ \myprod \  ( \, x^{\staralg}_{i} \, \myprod \, A \, \myprod \, e_{F} \, ) \ \ .
\endaligned
$$

\smallskip

\noindent{So}, \ $( a , a )_{{\CalH}_{F}} \, \ge \, 0$ for all $a \in {\myhecke{E}{F}}$.  \ That $( b , b )_{{\CalH}_{E}} \, \ge \, 0$ for any $b \, \in \, {\myhecke{F}{E}}$ is obvious.  Hence,

\begin{prop}\label{extension-prop} \ Suppose $F$ is a chamber and $E \subset F$ is a facet.  Suppose $\circ$ is an anti-involution of ${\CalH}_{F}$ satisfying \eqref{circ} and $\circ$ is extended to ${\CalH}_{E}$.   Then, the equivalence of categories in Theorem \ref{main0} preserves hermitan and unitarity modules.
\end{prop}

\vskip 0.30in

%%%%%%%%%%%%%%%%%%%%%%%%%%%%%%%%%%%%%%%%%%%%%%%%%%%%%%%%%%%%%%%%%%%%%%%%%
%%%%%%%%%%%%%%%%%%%%%%%%%%%%%%%%%%%%%%%%%%%%%%%%%%%%%%%%%%%%%%%%%%%%%%%%%
%%%%%%%%%%%%%%%%%%%%%%%%%%%%%%%%%%%%%%%%%%%%%%%%%%%%%%%%%%%%%%%%%%%%%%%%%

\section{Generalizations}

\smallskip

\subsection{Finite field groups} \quad We consider a finite field $\bF$ (with $q$ elements), and a connected reductive group $\bG$ defined over $\bF$.  Let $G$ be the group of $\bF$-rational points, and let $P=MU$ ($U$ the radical, $M$ a Levi factor) be the $\bF$-rational points of a parabolic subgroup defined over $\bF$.

\smallskip

\begin{thm} \ \ {\text{\rm{\cite{HC}}}} \  Take $G$ and $P=MU$ as above.  Suppose $\sigma$ and $\tau$ are irreducible cuspidal representrations of $M$.  The following are equivalent:

\begin{itemize}\label{fin-field-group-a}
\item[(i)] \ There exists $n \in N_{G}(M)$ so that $\Ad (n) \, \sigma \, = \, \tau$.
\smallskip
\item[(ii)] \ Suppose $(\lambda , V_{\lambda})$ is an irreducible representation of $G$ and $(V_{\lambda})_{U}$ is the $U$-covariants (a representation of $M$).  Then,  $(V_{\lambda})_{U}$ contains $\sigma$ if and only if it contains $\tau$.
\end{itemize}
\end{thm}

\medskip

Theorem \ref{fin-field-group-a} gives an equivalence relation on the set ${\mathcal T}$ of irreducible cuspidal representations of $M$.  For such a representation $\tau$, let $\Delta (\tau)$ denote the equivalence class of $\tau$.  Set
$$
X_{MU} \ := \ \{ \ \lambda \ \in \ \widehat{G} \ | \ (V_{\lambda})_{U} \ \ {\text{\rm{contains a cuspidal representation of $M$}}} \ \} \ .
$$
\noindent Then Theorem \ref{fin-field-group-a} also gives an equivalence relation on $X_{MU}$ as:
$$
\lambda_{1} \, , \, \lambda_{2} \, \in \, \widehat{G} \ \ : \ \ \lambda_{1} \ \sim \ \lambda_{2} \ \ {\text{\rm{if $(V_{\lambda_{1}})_{U}$ and $(V_{\lambda_{2}})_{U}$ share an irreducible cuspidal representation of $M$ }}} \ .
$$

\noindent Theorem \ref{fin-field-group-a} obviously provides a natural  bijjection between the equivalence classes of ${\mathcal T}$ and those of $X_{MU}$.

\medskip

We take $\sigma$ to be an irreducible cuspidal representation of $M$, and denote by $\Delta$ its equivalence class $\Delta$ in ${\mathcal T}$, and by $\Xi$ the corresponding equivalence class of representaions of $G$.  We define idempotent elements $e_{\sigma}$, $e_{\Delta}$ and $e_{\Xi}$ in the group algebra  $\bC G$ as follows:

$$
\aligned
e_{\sigma} (g) \ :&= \ \begin{cases} 
&{\frac{1}{\# (MU)}} \ \ \deg (\tau ) \ \Theta_{\sigma}  (m) {\hskip 0.53in}{\text{\rm{if $g = mu \ \in MU$}}}  \\ 
&{\ \ }0 {\hskip 1.86in}{\text{\rm{if $g \notin MU$}}} \  \\
\end{cases} \\
&\ \\
e_{\Delta} (g) \ :&= \ \begin{cases} 
&{\frac{1}{\# (MU)}} \ \ {\underset {\tau \in {\Delta}} \sum} \ \deg (\tau ) \ \Theta_{\tau}  (m) {\hskip 0.27in}{\text{\rm{if $g = mu \ \in MU$}}}  \\ 
&{\ \ }0 {\hskip 1.86in}{\text{\rm{if $g \notin MU$}}} \  \\
\end{cases} \\
&\ \\
e_{\Xi} \ :&= \ {\frac{1}{\# (G)}} \ \ {\underset {\lambda \in \Xi} \sum} \ \deg (\lambda ) \ \Theta_{\lambda}    \\ 
\endaligned
$$

\noindent The element $e_{\Xi}$ is the central idempotent in the group algebra $\bC G$ and for any irreducible representation $(\lambda , V_{\lambda})$ of $G$:  
$$
\lambda ( \, e_{\Xi} \, ) \ = \ \begin{cases}
&{\myId}_{V_{\lambda}} {\hskip 0.55in}{\text{\rm{if $\lambda \in \Xi$}}} \\
&{0}_{V_{\lambda}} {\hskip 0.53in}{\text{\rm{if $\lambda \notin \Xi$}}} \ .
\end{cases}
$$

\noindent The idempotent $e_{\sigma}$ is clearly the product (in any order) of the two idempotents 
$$
e_{a} (g) \ := \ \begin{cases} 
&{\frac{1}{\# (M)}} \ \ \deg (\sigma ) \ \Theta_{\sigma}  (m) {\hskip 0.30in}{\text{\rm{if $g = m \ \in \ M$}}}  \\ 
&{\ \ }0 {\hskip 1.54in}{\text{\rm{if $g \notin \ M$}}} \ , \\
\end{cases}
$$
\noindent and 
$$
e_{U} (g) \ := \ \begin{cases} 
&{\frac{1}{\# (U)}} \ \ {\underset {\tau \in {\Delta}} \sum} \ 1 {\hskip 0.62in}{\text{\rm{if $g \ \in \ U$}}}  \\ 
&{\ \ }0 {\hskip 1.22in}{\text{\rm{if $g \notin U$}}} \ , \\
\end{cases} 
{\hskip 0.68in}
$$

\noindent and similarly for $e_{\Delta}$.  For any irreducible representation $(\lambda, V_{\lambda})$ of $G$, we have  
$$
\lambda (e_{\Delta}) \ = \ \lambda (e_{a}) \circ \lambda (e_{U}) \ , 
$$
\noindent where $\lambda (e_{U})$ projects to the $U$-invariants of $V_{\lambda}$ (which we can identify with the $U$-covariants), and then the action of $\lambda (e_{a})$ on the $U$-invariants is projection to the isotypical component arising from $\sigma$.  Obviously, for any $\lambda \in \widehat{G}$, we have 
$\lambda (e_{\Xi} {\,}\myprod{\,} e_{\sigma}) \ = \ \lambda (e_{\sigma}) \  = \  \lambda (e_{\sigma} {\,}\myprod{\,} e_{\Xi})$.  This means the operator Fourier transforms of the three functions $e_{\Xi} {\,}\myprod{\,} e_{\sigma}$, $e_{\sigma}$ and $e_{\sigma} {\,}\myprod{\,} e_{\Xi}$ are equal.  This means 
\begin{equation}\label{compatibility-a}
e_{\Xi} {\,}\myprod{\,} e_{\sigma} \ = \ e_{\sigma} \ = \ e_{\sigma} {\,}\myprod{\,} e_{\Xi} \ .
\end{equation}

\noindent In a completely analogous way, 

\begin{equation}\label{compatibility-b}
e_{\Xi} {\,}\myprod{\,} e_{\Delta} \ = \ e_{\Delta} \ = \ e_{\Delta} {\,}\myprod{\,} e_{\Xi} \ .
\end{equation}

\noindent Define

\begin{equation}\label{hecke-alg-finite-b}
{\CalH}_{\sigma} := e_{\sigma} \myprod \bC (G) \ {\myprod} \ e_{\sigma}  \ \ , \ \ {\CalH}_{\Delta} := e_{\Delta} \myprod \bC (G) \ {\myprod} \ e_{\Delta} \ \ , \ \ {{\CalH}_{\Xi}} := e_{\Xi} \myprod \bC G \ {\myprod} \ e_{\Xi}  \ \ , \\
\end{equation}

\noindent and 

\begin{equation}\label{hecke-alg-finite-c}
\aligned
{\myhecke{\Xi}{\sigma}} \ :&= \ e_{\Xi} \ {\myprod} \bC G \ {\myprod} \ e_{\sigma} \quad &&{,} \qquad {\myhecke{\sigma}{\Xi}} := e_{\sigma} \ {\myprod} \ \bC G \ {\myprod} \ e_{\Xi} \\
{\myhecke{\Xi}{\Delta}} \ :&= \ e_{\Xi} \ {\myprod} \bC G \ {\myprod} \ e_{\Delta} \quad &&{,} \qquad {\myhecke{\Delta}{\Xi}} := e_{\Delta} \ {\myprod} \ \bC G \ {\myprod} \ e_{\Xi} \\
\endaligned
\end{equation}

\noindent The relations in \eqref{compatibility-a} and \eqref{compatibility-b} mean ${\CalH}_{\sigma}$ and ${\CalH}_{\Delta}$ are subalgebras of ${\CalH}_{\Xi}$.  The proof of Proposition \ref{key-prop} can be easily modified and combined with the referenced results on Morita equivalence to show:

\begin{prop}\label{key-prop-b} \ \ The  ideals  {\,}$\bC G {\,}{\myprod}{\,} e_{\sigma} {\,}{\myprod}{\,}  \bC G$, and {\,}$\bC G {\,}{\myprod}{\,} e_{\Delta} {\,}{\myprod}{\,}  \bC G${\,} of {\,}$\bC G${\,} satisfy:

\smallskip

\begin{itemize}

\item[(i)] \ Each equals ${\CalH}_{\Xi}$, i.e., the idempotents $e_{\sigma}$ and $e_{\Delta}$ are full idempotents of ${\CalH}_{\Xi}$.

\smallskip

\item[(ii.1)] \  ${\CalH}_{\Xi} \ = \ {\myhecke{\Xi}{\sigma}} {\,}\myprod{\,} {\myhecke{\sigma}{\Xi}}$, and \ ${\CalH}_{\sigma} \ = \ {\myhecke{\sigma}{\Xi}} {\,}\myprod{\,} {\myhecke{\Xi}{\sigma}}$.

\smallskip

\item[(ii.2)] \  ${\CalH}_{\Xi} \ = \ {\myhecke{\Xi}{\Delta}} {\,}\myprod{\,} {\myhecke{\Delta}{\Xi}}$, and \ ${\CalH}_{\Delta} \ = \ {\myhecke{\Delta}{\Xi}} {\,}\myprod{\,} {\myhecke{\Xi}{\Delta}}$.

\smallskip

\item[(iii)] \  The algebras ${\CalH}_{\sigma}$ and ${\CalH}_{\Delta}$ are Morita equivalent to ${\CalH}_{\Xi}$. 
\end{itemize}
\end{prop}

\medskip

\subsection{Local field groups} \quad We now consider $k$ a non-archimedean local field with notation as in the introduction.  Suppose $F$ is a facet of the building $\ScptB (\Gk )$, and $F$ is a facet with a subfacet $E$ and $\Gk_{F}$ and $\Gk_{E}$ the corresponding parahoric subgroups (so $\Gk^{+}_{E} \subset \Gk^{+}_{F} \subset \Gk_{F} \subset \Gk_{E}$). Then $G = \Gk_{E}/\Gk^{+}_{E}$ is the $\bF$-rational points of a reductive group defined over $\bF$ and $P = Gk_{F}/\Gk^{+}_{E}$ is a parabolic subgroup.  Let $MU$ be a Levi decompostion of $P$ and let $\sigma$ be an irreducible cuspidal representation of $M$.  Define $\Delta = \Delta (\sigma)$ and $\Xi$ as in the previous section.  The inflation of the idempotent $e_{\sigma}$ of $G$ to $\Gk_{E}$ obviously has support in $\Gk_{F}$.  For convenience of notation we continue to use the notation $e_{\sigma}$ to denote the inflation.  Denote by $e_{F}$ and $e_{E}$ respectively, the inflations of $e_{\Delta}$ and $e_{\Xi}$ to $\Gk_{E}$.  The support of $e_{F}$ is in $\Gk_{F}$.   

\medskip

Define
\begin{equation}\label{hecke-alg-finite-e}
{\CalH}_{\sigma} := e_{\sigma} \myprod \Cic (\Gk ) \ {\myprod} \ e_{\sigma}  \ \ , \ \ {\CalH}_{F} := e_{F} \myprod \Cic (\Gk ) \ {\myprod} \ e_{F} \ \ , \ \ {{\CalH}_{E}} := e_{E} \myprod \Cic \ {\myprod} \ e_{E}  \ \ , \\
\end{equation}

\noindent and 
\begin{equation}\label{hecke-alg-finite-f}
\aligned
{\myhecke{E}{\sigma}} \ :&= \ e_{E} \ {\myprod} \ \Cic (\Gk ) \ {\myprod} \ e_{\sigma} \quad &&{,} \qquad {\myhecke{\sigma}{E}} := e_{\sigma} \ {\myprod} \ \Cic( \Gk ) \ {\myprod} \ e_{E} \\
{\myhecke{E}{\Delta}} \ :&= \ e_{E} \ {\myprod} \ \Cic (\Gk ) \ {\myprod} \ e_{\Delta} \quad &&{,} \qquad {\myhecke{\Delta}{\Xi}} := e_{\Delta} \ {\myprod} \ \Cic ( \Gk ) \ {\myprod} \ e_{E} \\
\endaligned
\end{equation}

In an enitrely analogous fashion to Theorem \ref{main0} we have

\begin{thm}\label{main1}
The idempotents $e_{\sigma}$ and $e_{F}$ are full idempotents of the algebra ${\CalH}_{E}$, and so the algebras ${\CalH}_{\sigma}$ and ${\CalH}_{F}$ are Morita equivalent to  ${\CalH}_{E}$.
\end{thm}

\bigskip

The $\staralg$-anti-involution \ $f^{\staralg}(g) = \overline{f(g^{-1})}$ on $\Cic (\Gk )$ restricts to a $\staralg$-anti-involution on the algebras ${\CalH}_{\sigma}$, ${\CalH}_{F}$ and ${\CalH}_{E}$.   In analogy with Proposition \ref{extension-prop}:

\begin{prop} Suppose $\circ$ is an anti-involution of ${\CalH}_{F}$ which satisfies  \eqref{circ}: \ $\forall \ f \, \in \, e_{F} {\,}\myprod{\,} \Cic (\Gk_{E}) {\,}\myprod{\,} e_{F} \ \ : \ \ f^{\circ} \, = \, f^{\staralg}$.  Then there is an extension of $\circ$ to an anti-involution of ${\CalH}_{E}$, so that Morita equivalence of ${\CalH}_{F}$ and ${\CalH}_{F}$ preserves hermitian and unitary modules.  The same holds if we replace ${\CalH}_{F}$ by ${\CalH}_{\sigma}$.
\end{prop}

\vskip 0.20in

%%%%%%%%%%%%%%%%%%%%%%%%%%%%%%%%%%%%%%%%%%%%%%%%%%%%%%%%%%%%%%%%%%%%%%%%%
%%%%%%%%%%%%%%%%%%%%%%%%%%%%%%%%%%%%%%%%%%%%%%%%%%%%%%%%%%%%%%%%%%%%%%%%%
%%%%%%%%%%%%%%%%%%%%%%%%%%%%%%%%%%%%%%%%%%%%%%%%%%%%%%%%%%%%%%%%%%%%%%%%%

\section{Acknowledgments}

\medskip

Parts of this work was done during a visit of the first author to the HKUST Mathematics Department in Spring 2014, a visit of the authors to the University of Utah Mathematics Department in Summer 2015, and visits of the second author to the Cornell University Mathematics Department in Fall 2017, and the University of Oxford Mathematical Institute in Spring 2018.  The Departments and Institute are thanked for their hospitality.

%%%%%%%%%%%%%%%%%%%%%%%%%%%%%%%%%%%%%%%%%%%%%%%%%%%%%%%%%%%%%%%%%%%%%%%%%%%
%%%%%%%%%%%%%%%%%%%%%%%%%%%%%%%%%%%%%%%%%%%%%%%%%%%%%%%%%%%%%%%%%%%%%%%%%%%
%%%%%%%%%%%%%%%%%%%%%%%%%%%%%%%%%%%%%%%%%%%%%%%%%%%%%%%%%%%%%%%%%%%%%%%%%%%
%%%%%%%%%%%%%%%%%%%%%%%%%%%%%%%%%%%%%%%%%%%%%%%%%%%%%%%%%%%%%%%%%%%%%%%%%%%
%%%%%%%%%%%%%%%%%%%%%%%%%%%%%%%%%%%%%%%%%%%%%%%%%%%%%%%%%%%%%%%%%%%%%%%%%%%

\vskip 0.30in

%%%%%%%%%%%%%%%%%%%%%%%%%%%%%%%%%%%%%%%%%%%%%%%%%%%%%%%%%%%%%%%%%%%%%%%%%
%%%%%%%%%%%%%%%%%%%%%%%%%%%%%%%%%%%%%%%%%%%%%%%%%%%%%%%%%%%%%%%%%%%%%%%%%
%%%%%%%%%%%%%%%%%%%%%%%%%%%%%%%%%%%%%%%%%%%%%%%%%%%%%%%%%%%%%%%%%%%%%%%%%

\ifx\undefined\bysame
\newcommand{\bysame}{\leavevmode\hbox to3em{\hrulefill}\,}
\fi

%%%%%%%%%%%%%%%%%%%%%%%%%%%%%%%%%%%%%%%%%%%%%%%%%%%%%%%%%%%%%%%%%%%%%%%%%
%%%%%%%%%%%%%%%%%%%%%%%%%%%%%%%%%%%%%%%%%%%%%%%%%%%%%%%%%%%%%%%%%%%%%%%%%
%%%%%%%%%%%%%%%%%%%%%%%%%%%%%%%%%%%%%%%%%%%%%%%%%%%%%%%%%%%%%%%%%%%%%%%%%
\vfill
{\small{
{\tsc{Department of Mathematics, Malott Hall, Cornell University, Ithaca, NY 14853--0099, USA, \\
Email:}}{\tt{barbasch{\char'100}math.cornell.edu}}

\vskip 0.10in

{\tsc{Department of Mathematics, The Hong Kong University of Science and Technology, Clear Water Bay Road, Hong Kong, Email:}}{\tt{amoy{\char'100}ust.hk}}

}}

\eject

%%%%%%%%%%%%%%%%%%%%%%%%%%%%%%%%%%%%%%%%%%%%%%%%%%%%%%%%%%%%%%%%%%%%%%%%%
%%%%%%%%%%%%%%%%%%%%%%%%%%%%%%%%%%%%%%%%%%%%%%%%%%%%%%%%%%%%%%%%%%%%%%%%%
%%%%%%%%%%%%%%%%%%%%%%%%%%%%%%%%%%%%%%%%%%%%%%%%%%%%%%%%%%%%%%%%%%%%%%%%%
\vfill \eject
%%%%%%%%%%%%%%%%%%%%%%%%%%%%%%%%%%%%%%%%%%%%%%%%%%%%%%%%%%%%%%%%%%%%%%%%%
%%%%%%%%%%%%%%%%%%%%%%%%%%%%%%%%%%%%%%%%%%%%%%%%%%%%%%%%%%%%%%%%%%%%%%%%%
%%%%%%%%%%%%%%%%%%%%%%%%%%%%%%%%%%%%%%%%%%%%%%%%%%%%%%%%%%%%%%%%%%%%%%%%%
 
\end{document}